\documentclass{article}

\usepackage{amsmath,amssymb,amsthm,amsfonts,mathrsfs,latexsym,mathtools,bm}
\usepackage[abbrev]{amsrefs}
\usepackage{stmaryrd}	
\usepackage[OT2,T1]{fontenc}
\usepackage[russian,english]{babel}
\usepackage{cleveref}
\usepackage{autonum}
\usepackage[marginparwidth=0pt,margin=25truemm]{geometry}
\usepackage{enumitem}

\theoremstyle{plain}
\newtheorem{thm}{Theorem}[section]
\newtheorem{lem}[thm]{Lemma}
\newtheorem{prop}[thm]{Proposition}
\newtheorem{cor}[thm]{Corollary}

\theoremstyle{definition}

\newtheorem{dfn}[thm]{Definition}
\newtheorem{rem}[thm]{Remark}

\renewcommand{\labelenumi}{(\roman{enumi})}

\allowdisplaybreaks[2]

\everymath{\displaystyle}

\font\sevency=wncyr7
\newcommand{\sha}{ {\,\hbox{\sevency X}\, } }
\newcommand{\x}{ { \hbox{\sevency X} } }

\newcommand{\rising}[2]{\raisebox{{#1}pt}{{#2}}}				

\newcommand{\wt}{\mathrm{wt}}

\newcommand{\QQ}{\mathbb{Q}}
\newcommand{\ZZ}{\mathbb{Z}}
\newcommand{\RR}{\mathbb{R}}

\newcommand{\CC}{\mathbb{C}}
\newcommand{\HH}{\mathbb{H}}

\newcommand{\ha}{\mathfrak{H}}
\newcommand{\relmid}{\mathrel{}\middle|\mathrel{}}

\newcommand{\xx}{\mathrm{x}}
\newcommand{\yy}{\mathrm{y}}
\newcommand{\IT}{\widetilde{\mathrm{I}}}
\newcommand{\II}{\mathrm{I}}

\newcommand{\eps}{\varepsilon}
\newcommand{\tsha}{\widetilde{\sha}}
\newcommand{\tas}{\widetilde{*}}
\newcommand{\vvec}[3]{\bigg(\begin{matrix}#1\\#2\end{matrix};#3\bigg)}

\title{Shuffle regularization for multiple Eisenstein series of level $N$}
\author{Hayato Kanno}
\date{}

\begin{document}

\maketitle
\begin{abstract}
In 2018, Bachmann and Tasaka discovered a relationship between the Fourier expansion of multiple Eisenstein series of level $1$ and the Goncharov coproduct on formal iterated integrals corresponding to multiple zeta values. They also constructed shuffle regularized multiple Eisenstein series of level $1$, which satisfy the shuffle relations similar to multiple zeta values. In this paper, we expand their results to arbitrary level and give some linear relations among multiple Eisenstein series of level $N$.
\end{abstract}

\tableofcontents

\section{Introduction}\label{section:intro}

A multiple zeta value (MZV for short) is a generalization of the Riemann zeta value and is defined for integers $n_1,\dots,n_{r-1}\ge1$ and $n_r\ge2$ by
$$\zeta(n_1,\dots,n_r)=\sum_{0<k_1<\cdots<k_r}\frac{1}{k_1^{n_1}\cdots k_r^{n_r}}.$$
MZV also have an iterated integral representation. These two different representations give $\QQ$-linear relations among MZVs, which are called \textit{double shuffle relations}. For example, the product of $\zeta(2)$ and $\zeta(3)$ can be expanded in two ways as
\begin{align}
\zeta(2)\zeta(3)&=\zeta(2,3)+\zeta(3,2)+\zeta(5)\\
&=3\zeta(2,3)+\zeta(3,2)+6\zeta(1,4),
\end{align}
from which we deduce $$\zeta(5)=2\zeta(2,3)+6\zeta(1,4).$$
It is known that there are numerous $\QQ$-linear relations among MZVs. Zagier (\cite{Zag92}, \cite{Zag93}) initially noted that there exist some relations among double zeta values whose coefficients are originated from modular forms for $\mathrm{SL}_2(\ZZ)$. He also conjectured the dimension of the space spanned by double zeta values of $\zeta(a,k-a)$ $(a:\mathrm{odd},2\le a\le k/2)$, which implies that the number of linearly independent $\QQ$-linear relations among them is equal to $\mathrm{dim}_{\CC}\mathcal{M}_k(\mathrm{SL}_2(\ZZ))$, the dimension of the space of modular forms of weight $k$ for $\mathrm{SL}_2(\ZZ)$. 

A formulation of modular relations for double zeta values was first established by Gangl, Kaneko and Zagier \cite{GKZ}. It uses the space of even period polynomials, which by the theory of Eichler-Shimura is isomorphic to the space of modular forms for $\mathrm{SL}_2(\ZZ)$. They also introduced double Eisenstein series $G_{r,s}(\tau)$ defined for integers $r,s\ge2$ as a holomorphic function on the upper half plane $\HH$ by
$$G_{r,s}(\tau)\coloneq\sum_{\substack{0\prec\lambda_1\prec\lambda_2\\\lambda_1,\lambda_2\in\ZZ\tau+\ZZ}}\frac{1}{\lambda_1^{r}\lambda_2^s},$$
where the order of lattice points $l_1\tau+m_1\succ l_2\tau+m_2$ means $l_1>l_2$ or $l_1=l_2,m_1>m_2$. They constructed regularized double Eisenstein series $G_{r,s}(q)$ defined for integers $r,s\ge1$ as a $q$-series, which satisfies the extended double shuffle relations. Double Eisenstein series are not modular forms in general, but they are expected to play an important role for modular relations among MZVs.
Tasaka (\cite{Tasaka20}) gave an explicit formula for decomposing a Hecke eigenform into double Eisenstein series. An example of such a decomposition is the following:
$$22680G^{\frac{1}{2}}_{9,3}(q)-35364G^{\frac{1}{2}}_{7,5}(q)-29145G^{\frac{1}{2}}_{5,7}(q)+13006G^{\frac{1}{2}}_{3,9}(q)+22680G^{\frac{1}{2}}_{1,11}(q)=\frac{1}{680}\Delta(q).$$
Here, we denote $G^{\frac{1}{2}}_{r,s}(q)=G_{r,s}(q)+\frac{1}{2}G_{r+s}(q)$ and $\Delta(q)=q\prod_{n>0}(1-q^n)^{24}$.

In \cite{BT}, Bachmann and Tasaka studied multiple Eisenstein series (MES for short) $G_{n_1,\dots,n_r}(\tau)$ for general depth $r>0$, which are defined for $n_1,\dots,n_r\ge2$ as a generalization of double Eisenstein series. They revealed the relationship between the Fourier expansion of MES and the Goncharov coproduct on iterated integrals. They also constructed shuffle regularized MES $G^\sha_{n_1,\dots,n_r}(q)$ as $q$-series, which satisfies restricted double shuffle relations. For example, we have $$G_5(q)=2G_{2,3}(q)+6G^\sha_{1,4}(q).$$

In this paper, we expand their results to general level. Since MES can be viewed as a multivariate generalization of the classical Eisenstein series, we can consider MES with levels as a generalization of classical Eisenstein series for congruence subgroups. Kaneko and Tasaka (\cite{KTa}) considered double Eisenstein series of level 2. They provided the relationship between double zeta values of level 2 and modular forms of level 2, and obtained the analogous results of Gangl--Kaneko--Zagier \cite{GKZ} and Kaneko \cite{Kaneko04}. Kina (\cite{Kina24}) considered double Eisenstein series of level 4, whose constant term of its Fourier expansion is a double $\widetilde{T}$-value introduced in Kaneko--Tsumura \cite{KTs}, and obtained the analogous results of Gangl--Kaneko--Zagier \cite{GKZ} and Kaneko \cite{Kaneko04}. For general level, Yuan and Zhao (\cite{YZ15}) studied double Eisenstein series of level $N$ and gave the regularization such that they satisfy extended double shuffle relations.

Let $N\in\ZZ_{>0}$ and $\eta=\exp(2\pi\sqrt{-1}/N)$. Let $\ha^1$, $\widetilde{\ha^0}$ and $\ha^2$ be non-commutative polynomial rings of double indices defined by
\begin{align}
\ha^1&\coloneq\QQ\left\langle\binom{n}{a}\relmid n\in\ZZ_{\ge1},a\in\ZZ/N\ZZ\right\rangle,\\
\widetilde{\ha^0}&\coloneq\QQ+\sum_{n\ge2,a\in\ZZ/N\ZZ}\ha^1\binom{n}{a},\\
\ha^2&\coloneq\QQ\left\langle\binom{n}{a}\relmid n\in\ZZ_{\ge2},a\in\ZZ/N\ZZ\right\rangle.
\end{align}
Note that $\widetilde{\ha^0}$, $\ha^2$ are the space of admissible indices for MZV and MES of level $N$, respectively. We can equip $\ha^1$ two different commutative and associative algebra structures with the harmonic product $\tas$ and the shuffle product $\tsha$. We will state the precise definition of these products in \Cref{subsec:MZV}. For $r\ge2$, $n_1,\dots,n_r\ge1$ and $a_1,\dots,a_r\in\ZZ/N\ZZ$, let $\binom{n_1,\dots,n_r}{a_1,\dots,a_r}$ denote the concatenation $\binom{n_1}{a_1}\cdots\binom{n_r}{a_r}$. 
\begin{dfn}
We define MES of level $N$ as a holomorphic function on upper half plane $\HH$ for integers $n_1,\dots,n_r\ge2$ and $a_1,\dots,a_r\in\ZZ/N\ZZ$ by
$$G\bigg(\begin{matrix}n_1,\dots,n_r\\a_1,\dots,a_r\end{matrix};\tau\bigg)
\coloneq\lim_{L\to\infty}\lim_{M\to\infty}\sum_{\substack{0\prec l_1N\tau+m_1\prec\dots\prec l_rN\tau+m_r\\l_i\in\ZZ_L,m_i\in\ZZ_M,m_i\equiv a_i\hspace{-8pt}\pmod{N}}}\frac{1}{(l_1N\tau+m_1)^{n_1}\dots(l_rN\tau+m_r)^{n_r}}.$$
Here, $\ZZ_L=\{l\in\ZZ\mid\left|l\right|\le L\}$. The order of lattice points is defined by
$$l_1\tau+m_1\succ l_2\tau+m_2\xLeftrightarrow{\mathrm{def}}l_1>l_2\quad\text{or}\quad\begin{cases}l_1=l_2\\m_1>m_2\end{cases}.$$
\end{dfn}
Some of classical Eisenstein series can be expressed via this function. For instance, for $k\ge3$ and $a\in(\ZZ/N\ZZ)\setminus\{0\}$, the sum
$$G\vvec{k}{a}{\tau}+(-1)^kG\vvec{k}{-a}{\tau}=\!\!\!\!\sum_{\substack{l,m\in\ZZ\\l\equiv 0, m\equiv a\hspace{-3pt}\pmod{N}}}\!\!\!\!\frac{1}{l\tau+m}$$
is a holomorphic modular form for $\Gamma_1(N)$ of weight $k$.\color{black}\\
We often regard $G$ as the $\QQ$-linear map $G:\ha^2\to\mathcal{O}(\HH)$ by sending a word $\binom{n_1,\dots,n_r}{a_1,\dots,a_r}$ to $G\vvec{n_1,\dots,n_r}{a_1,\dots,a_r}{\tau}$ and an enpty word to $G(\emptyset;\tau)=1$. For instance,
$$G\left(\binom{n_1}{a_1}+2\binom{n_2,n_3}{a_2,a_3}\right)=G\vvec{n_1}{a_1}\tau+2G\vvec{n_2,n_3}{a_2,a_3}\tau.$$
MES of arbitrary level are studied in Yuan--Zhao \cite{YZ15} in the case of depth 2.
The Fourier expansion can be expressed with MZVs and multiple divisor function of level $N$. For instance, it holds
\begin{align}\label{eq:fou}
G\bigg(\begin{matrix}2,3\\1,1\end{matrix};\tau\bigg)=\zeta\binom{2,3}{1,1}+\zeta\binom{2}{0}g\bigg(\begin{matrix}3\\1\end{matrix};q\bigg)+\zeta\binom{2}{1}g\bigg(\begin{matrix}3\\1\end{matrix};q\bigg)+3\zeta\binom{3}{0}g\bigg(\begin{matrix}2\\1\end{matrix};q\bigg)+g\bigg(\begin{matrix}2,3\\1,1\end{matrix};q\bigg).
\end{align}
where $q=e^{2\pi\sqrt{-1}\tau}$. Here, MZV of level $N$ are defined for integers $n_1,\dots,n_{r-1}\ge1$, $n_r\ge2$ and $a_1,\dots,a_r\in\ZZ/N\ZZ$ by
$$\zeta\binom{n_1,\dots,n_r}{a_1,\dots,a_r}\coloneq\sum_{\substack{0<k_1<\dots<k_r\\\forall i,k_i\equiv a_i\hspace{-8pt}\pmod{N}}}\frac{1}{k_1^{n_1}\cdots k_r^{n_r}}\in\RR,$$
and the multiple divisor functions of level $N$ are defined for integers $n_1,\dots,n_r\ge1$ and $a_1,\dots,a_r\in\ZZ/N\ZZ$ by
$$g\bigg(\begin{matrix}n_1,\dots,n_r\\a_1,\dots,a_r\end{matrix};q\bigg)\coloneq\left(\frac{-2\pi\sqrt{-1}}{N}\right)^n\sum_{\substack{0<d_1<\cdots<d_r\\c_1,\dots,c_r>0}}\prod_{i=1}^r\frac{\eta^{a_ic_i}c_i^{n_i-1}}{(n_i-1)!}q^{c_id_i}\in\CC\llbracket q\rrbracket,$$
where $n=n_1+\cdots+n_r$. Let $\zeta:\widetilde{\ha^0}\to\RR$ and $g:\ha^1\to\CC\llbracket q\rrbracket$ be $\QQ$-linear maps, which send a word $w$ to $\zeta(w)$ and $g(w;q)$, respectively and $\zeta(\emptyset)=g(\emptyset;q)=1$.

Goncharov considered the algebra generated by formal iterated integrals $\II(a_0;a_1,\dots,a_m;a_{m+1})$, which correspond to iterated integrals
$$\int_{a_0}^{a_{m+1}}\frac{dt}{t-a_1}\cdots\frac{dt}{t-a_m}\coloneq\int_{a_0<t_1<\cdots<t_m<a_{m+1}}\frac{dt_1}{t_1-a_1}\cdots\frac{dt_m}{t_m-a_m}.$$
He proved the algebra has a Hopf algebra structure with the Goncharov coproduct $\Delta$. Let $\widetilde{\mathcal{I}^1}$ be a $\QQ$-algebra spanned by formal iterated integrals corresponding to MZVs of level $N$. $\widetilde{\mathcal{I}^1}$ is also a Hopf algebra with a coproduct $\Delta_\mu$ induced from Goncharov coproduct $\Delta$. We will state the precise definition of $\Delta_\mu$ and $\widetilde{\mathcal{I}^1}$ in \Cref{subsection:FMZV} and \Cref{subsection:CompGon}, respectively. Since we know that $\widetilde{\mathcal{I}^1}$ is isomorphic to $(\ha^1,\tsha)$ as $\QQ$-algebras, we can equip $\ha^1$ with a Hopf algebra structure. For example we have
\begin{align}\label{eq:cop}
\Delta_\mu\left(\binom{2,3}{1,1}\right)=\binom{2,3}{1,1}\otimes1+\binom{2}{0}\otimes\binom{3}{1}+\binom{2}{1}\otimes\binom{3}{1}+3\binom{3}{0}\otimes\binom{2}{1}+1\otimes\binom{2,3}{1,1}.
\end{align}
Now, comparing equations (\ref{eq:fou}) and (\ref{eq:cop}), we can see the relationships between the Fourier expansion of MES of level $N$ and the Goncharov coproduct. The first main result of this paper is as follows.

\begin{thm}\label{thm:main1}
For any $w\in\ha^2$, we have
$$G(w;\tau)=(\zeta\star g)(w;q)\qquad(q=e^{2\pi\sqrt{-1}\tau}),$$
where $\zeta\star g=m\circ(\zeta\otimes g)\circ\Delta_\mu$ and $m$ is multiplication on $\CC\llbracket q\rrbracket$.
\end{thm}
By using this, we can construct shuffle regularized MES of level $N$. Kitada constructed shuffle regularized multiple divisor functions of level $N$, $g^{\tsha}:\ha^1\to\CC\llbracket q\rrbracket$. Then, we define shuffle regularized MES of level $N$ $G^{\tsha}:\ha^1\to\mathcal{O}(\HH)$ by
$$G^{\tsha}(w;\tau)\coloneq(\zeta^{\tsha}\star g^{\tsha})(w;q)$$
for $w\in\ha^1$, where $\zeta^{\tsha}:\ha^1\to\CC$ is the map of shuffle regularized MZVs of level $N$. 
It follows from the construction that $G^{\tsha}$ satisfy the shuffle product formula. We will prove that $G^{\tsha}=G$ on $\ha^2$, in other words, the regularized MES of level $N$ coincide with the original one in the convergent case. Then, considering the product of $G$'s yields the restricted double shuffle relations for the regularized MES of level $N$, which is the second main result of this paper.
\begin{thm}[Restricted double shuffle relation]\label{thm:main2}
For any words $w_1,w_2\in\ha^2$, we have
$$G(w_1\tas w_2;\tau)=G^{\tsha}(w_1\tsha w_2;\tau).$$
\end{thm}
We also give other linear relations among MES of level $N$. Kitada obtained distribution relations and sum and weighted sum formulas for double Eisenstein series of level $N$. We give the distribution relations for MES of level $N$.

The organization of this paper is as follows. In \Cref{sec:MZV}, we recall multiple $L$-values (MLVs) introduced by Goncharov \cite{Gon98} and Arakawa--Kaneko \cite{AK} and MZVs of level $N$ introduced by Yuan--Zhao \cite{YZ16}. Note that their shuffle and harmonic product are different in the algebra of indices $\ha$. We give the relationships between their shuffle and harmonic product. In \Cref{sec:MES}, we give an explicit formula of the Fourier expansion of MES of level $N$ in the same way as Bachmann--Tasaka \cite{BT}. In \Cref{sec:FMZV}, we study formal iterated integrals corresponding to MZVs of level $N$. The main point is to study the Goncharov coproduct on these iterated integrals and compare them to the Fourier expansion of MES of level $N$. At the end of \Cref{sec:FMZV}, we observe explicit coincidence between the Fourier expansion of MES of level $N$ and the Goncharov coproduct. Using this, we construct shuffle regularized MES of level $N$ in \Cref{sec:sha}. In \Cref{sec:rel}, we give some linear relations among regularized MES of level $N$.

\section*{Acknowledgment}

The author would like to express his deepest gratitude to his supervisor, Professor Yasuo Ohno of Tohoku University. He taught the author attitude to study mathematics and gave a great deal of advice. The author also would like to express his deepest gratitude to Professor Koji Tasaka of Kindai University for comments, corrections, and ideas on this paper. In particular, he gave the author very helpful advice for the argument in \Cref{subsection:FMZV}, and introduced the author thesis \cite{Kitada23} of his former student Toi Kitada. The author also would like to thank the members of Ohno Laboratory for their kindness on both academically and as friends.

\section{Multiple zeta values of level $N$}\label{sec:MZV}

Yuan--Zhao \cite{YZ16} introduced MZVs of level $N$ and constructed their regularizations via multiple $L$-values. In this section, we reconstruct their results in terms of the ``word algebra'', introduced by Arakawa--Kaneko \cite{AK}.

\subsection{Multiple $L$-values of level $N$}

Let $x,y_a$, $a\in\ZZ/N\ZZ$, be letters and denote $\ha$ as a non-commutative polynomial ring generated by $x,y_a$ and let $\ha^0$ be a $\QQ$-subalgebra of $\ha$ defined by
\begin{align}
\ha&\coloneq\QQ\left\langle x,y_a\relmid a\in\ZZ/N\ZZ\right\rangle,\\
\ha^0&\coloneq\QQ+\sum_{\substack{a\in\ZZ/N\ZZ\\a\ne0}}\QQ y_a+\sum_{a\in\ZZ/N\ZZ}y_a\ha x+\sum_{\substack{a,b\in\ZZ/N\ZZ\\b\ne0}}y_a\ha y_b.
\end{align}
We can identify $\ha^0$ with the algebra of double indices by sending $z_{n,a}=y_ax^{n-1}$ to the double index $\binom{n}{a}$:
$$\ha^0\simeq\QQ+\sum_{\substack{n\ge1,a\in\ZZ/N\ZZ\\(n,a)\ne(1,0)}}\ha^1\binom{n}{a}.$$
Hereinafter, we identify a word $z_{n,a}$ and a double index $\binom{n}{a}$.

\begin{dfn}[Goncharov \cite{Gon98}, Arakawa--Kaneko \cite{AK}]\label{dfn:L}
Define the multiple $L$-value (MLV) of shuffle and harmonic type $L_\sha,L_\ast:\ha^0\to\CC$ by
\begin{gather}
L_{\sha}\binom{n_1,\dots,n_r}{a_1,\dots,a_r}\coloneq\int_0^1\frac{\eta^{a_1}dt}{1-\eta^{a_1}t}\left(\frac{dt}{t}\right)^{n_1-1}\cdots\frac{\eta^{a_r}dt}{1-\eta^{a_r}t}\left(\frac{dt}{t}\right)^{n_r-1},\\
L_\ast\binom{n_1,\dots,n_r}{a_1,\dots,a_r}\coloneq\sum_{0<k_1<\cdots<k_r}\frac{\eta^{a_1k_1+\cdots+a_rk_r}}{k_1^{n_1}\cdots k_r^{n_r}}
\end{gather}
for any word $\binom{n_1,\dots,n_r}{a_1,\dots,a_r}\in\ha^0$, and $L_\sha(\emptyset)=L_\ast(\emptyset)=1$, together with $\QQ$-linearity.
\end{dfn}

MLVs of shuffle type also have a series expression:
$$L_{\sha}\binom{n_1,\dots,n_r}{a_1,\dots,a_r}=\sum_{0<k_1<\cdots<k_r}\frac{\eta^{(a_1-a_2)k_1+\cdots(a_{r-1}-a_r)k_{r-1}+a_rk_r}}{k_1^{n_1}\cdots k_r^{n_r}}.$$
Define shuffle product $\sha$ on $\ha$ recursively by
\begin{enumerate}
\renewcommand{\labelenumi}{(S\arabic{enumi})}
\item$w\sha1=1\sha w=w,$
\item$u_1w_1\sha u_2w_2=u_1(w_1\sha u_2w_2)+u_2(u_1w_1\sha w_2)$
\end{enumerate}
for any $u_1,u_2\in\{x,y_a\mid a\in\ZZ/N\ZZ\}$ and any words $w,w_1,w_2\in\ha$, together with $\QQ$-bilinearity.
Note that $\ha^1$ is generated by $z_{n,a}=y_ax^{n-1}$ $(n>0,a\in\ZZ/N\ZZ)$.
We also define the harmonic product $\ast$ on $\ha^1$ recursively by
\begin{enumerate}
\renewcommand{\labelenumi}{(H\arabic{enumi})}
\item$w\ast1=1\ast w=w,$
\item$z_{n_1,a_1}w_1\ast z_{n_2,a_2}w_2=z_{n_1,a_1}(w_1\ast z_{n_2,a_2}w_2)+z_{n_2,a_2}(z_{n_1,a_1}w_1\ast w_2)+z_{n_1+n_2,a_1+a_2}(w_1\ast w_2),$
\end{enumerate}
for any $z_{n_1,a_1},z_{n_2,a_2}$ and any words $w_1,w_2\in\ha^1$, together with $\QQ$-bilinearity.
It is known that $L_\sha$ and $L_\ast$ are algebra homomorphisms i.e.
$$L_\#(w_1\#w_2)=L_\#(w_1)L_\#(w_2)\quad(w_1,w_2\in\ha^0,\#\in\{\sha,\ast\}).$$

\begin{rem}
There exist regularized MLVs, which are defined as polynomials, $L_\#^\mathrm{reg}:(\ha^1,\#)\to\CC[T]$ and they are algebra homomorphisms (see Arakawa--Kaneko \cite{AK}).
\end{rem}

\subsection{Multiple zeta values of level N}\label{subsec:MZV}

\begin{dfn}[Yuan--Zhao \cite{YZ16}]
Define the multiple zeta value (MZV) of level $N$ $\zeta:\widetilde{\ha^0}\to\RR$ by
$$\zeta\binom{n_1,\dots,n_r}{a_1,\dots,a_r}\coloneq\sum_{\substack{0<k_1<\dots<k_r\\{}^\forall i,k_i\equiv a_i\hspace{-8pt}\pmod{N}}}\frac{1}{k_1^{n_1}\cdots k_r^{n_r}}$$
for integers $n_1,\dots,n_{r-1}\ge1$, $n_r\ge2$ and $a_1,\dots,a_r\in\ZZ/N\ZZ$, together with $\QQ$-linearity and $\zeta(\emptyset)=1$.
\end{dfn}

MZVs of level $N$ have an iterated integral representation:
$$\zeta\binom{n_1,\dots,n_r}{a_1,\dots,a_r}=\int_0^1\frac{t^{a_1-1}dt}{1-t^N}\left(\frac{dt}{t}\right)^{n_1-1}\frac{t^{a_2-a_1-1}dt}{1-t^N}\left(\frac{dt}{t}\right)^{n_2-1}\cdots\frac{t^{a_r-a_{r-1}-1}dt}{1-t^N}\left(\frac{dt}{t}\right)^{n_r-1}.$$
Define a $\QQ$-linear bijection $\rho:\ha^1\to\ha^1$ and a $\QQ(\eta)$-linear bijection $\pi:\ha^1\otimes_\QQ\QQ(\eta)\to\ha^1\otimes_\QQ\QQ(\eta)$ by
\begin{gather}
\rho(z_{n_1,a_1}\cdots z_{n_r,a_r})=z_{n_1,a_1}z_{n_2,a_2-a_1}\cdots z_{n_r,a_r-a_{r-1}},\\
\pi(z_{n,a})=N^{-1}\sum_{b\in\ZZ/N\ZZ}\eta^{-ab}z_{n,b},
\end{gather}
Define two products $\tsha,\tas:\ha^1\times\ha^1\to\ha^1$ by
\begin{gather}
w_1\tsha w_2\coloneq\rho^{-1}\left(\rho(w_1)\sha\rho(w_2)\right),\\
w_1\tas w_2\coloneq\pi^{-1}\left(\pi(w_1)\ast\pi(w_2)\right).
\end{gather}

MZVs of level $N$ can be written as $\QQ(\eta)$-linear combination of MLVs by using the linear maps $\rho$ and $\pi$.

\begin{prop}[Yuan--Zhao \cite{YZ16}]\label{prop:zetaL}
For $w\in\widetilde{\ha^0}$, we have
$$\zeta(w)=(L_\sha\circ\pi\circ\rho)(w)=(L_\ast\circ\pi)(w).$$
\end{prop}

It is known that MZVs of level $N$ satisfy the harmonic and shuffle products (see Yuan--Zhao \cite{YZ16}), not in the sence of the word algebra. Now, we rephrase these facts via the word algebra.
\begin{prop}[cf. Yuan--Zhao \cite{YZ16}]
$\zeta:(\widetilde{\ha^0},\tsha)\to\RR$ is an algebra homomorphism.
\end{prop}

\begin{proof}Let $y_a=\frac{t^{a-1}dt}{1-t^N}$ and $x=\frac{dt}{t}$. By the integral representation of MZV, it holds
$$\zeta(w)=\int_0^1\rho(w)$$
for any word $w\in\widetilde{\ha}^0$. For any words $w_1,w_2\in\widetilde{\ha}^0$, we have
\begin{align}
\zeta(w_1\tsha w_2)&=\int_0^1\rho(w_1\tsha w_2)=\int_0^1\rho(w_1)\sha\rho(w_2)\\
&=\left(\int_0^1\rho(w_1)\right)\left(\int_0^1\rho(w_2)\right)=\zeta(w_1)\zeta(w_2).
\end{align}
\end{proof}

\begin{prop}[cf. Yuan--Zhao \cite{YZ16}]
The $\tas$-product is well-defined and $\zeta:(\widetilde{\ha^0},\tas)\to\RR$ is an algebra homomorphism.
\end{prop}

\begin{proof}
It suffices to show the $\tas$-product is determined recursively by
\begin{enumerate}
\renewcommand{\labelenumi}{(T\arabic{enumi})}
\item$w\tas1=1\tas w=w,$
\item$z_{n_1,a_1}w_1\tas z_{n_2,a_2}w_2=z_{n_1,a_1}(w_1\tas z_{n_2,a_2}w_2)+z_{n_2,a_2}(z_{n_1,a_1}w_1\tas w_2)+\delta_{a_1,a_2}z_{n_1+n_2,a_1+a_2}(w_1\tas w_2)$
\end{enumerate}
for any $z_{n_1,a_1},z_{n_2,a_2}$ and any words $w_1,w_2$.
(T1) is clear by definition. We prove (T2) by induction on $l(w_1)+l(w_2)$. Here, $l(w)$ denote the length of the word $w$. When $l(w_1)+l(w_2)=0$, we have
\begin{align}
z_{n_1,a_1}\tas z_{n_2,a_2}&=\pi^{-1}\bigg(N^{-2}\sum_{b_1,b_2\in\ZZ/N\ZZ}\eta^{-a_1b_1-a_2b_2}(z_{n_1,b_1}z_{n_2,b_2}+z_{n_2,b_2}z_{n_1,b_1}+z_{n_1+n_2,b_1+b_2})\bigg)\\
&=z_{n_1,a_1}z_{n_2,a_2}+z_{n_2,a_2}z_{n_1,a_1}+\pi^{-1}\bigg(N^{-2}\sum_{b_1,b_2\in\ZZ/N\ZZ}\eta^{-a_1b_1-a_2b_2}z_{n_1+n_2,b_1+b_2}\bigg).
\end{align}
In the sum of the third term, replacing $b_1^\prime=b_1+b_2$, we have
\begin{align}
\pi^{-1}\bigg(N^{-2}\sum_{b_1,b_2\in\ZZ/N\ZZ}\eta^{-a_1b_1-a_2b_2}z_{n_1+n_2,b_1+b_2}\bigg)&=\pi^{-1}\bigg(N^{-2}\sum_{b_2\in\ZZ/N\ZZ}\eta^{(a_1-a_2)b_2}\sum_{b_1^\prime\in\ZZ/N\ZZ}\eta^{-a_1b_1^\prime}z_{n_1+n_2,b_1^\prime}\bigg)\\
&=\pi^{-1}\bigg(\delta_{a_1,a_2}N^{-1}\sum_{b_1^\prime\in\ZZ/N\ZZ}\eta^{-a_1b_1^\prime}z_{n_1+n_2,b_1^\prime}\bigg)=\delta_{a_1,a_2}z_{n_1+n_2,a_1}.
\end{align}
When $l(w_1)+l(w_2)>0$, we put $w_1=z_{n_2,a_2}\cdots z_{n_r,a_r}$ and $w_2=z_{n_{r+2},a_{r+2}}\cdots z_{n_{r+s},a_{r+s}}$. Using inductive hypothesis, we have
\begin{align}
z_{n_1,a_1}w_1\tas z_{n_{r+1},a_{r+1}}w_2&=\pi^{-1}\bigg(N^{-(r+s)}\sum_{b_1,\dots,b_{r+s}}\eta^{-\bm{a}\cdot\bm{b}}z_{n_1,b_1}\cdots z_{n_r,b_r}\ast z_{n_{r+1},b_{r+1}}\cdots z_{n_{r+s},b_{r+s}}\bigg)\\
&=z_{n_1,a_1}(w_1\tas z_{n_{r+1},a_{r+1}}w_2)+z_{n_{r+1},a_{r+1}}(z_{n_1,a_1}w_1\tas w_2)\\
&+\pi^{-1}\bigg(N^{-2}\sum_{b_1,b_{r+1}\in\ZZ/N\ZZ}\eta^{-a_1b_1-a_{r+1}b_{r+1}}z_{n_1+n_{r+1},b_1+b_{r+1}}\bigg)w_1\tas w_2,
\end{align}
where $\bm{a}=(a_1,\dots,a_r)$ and $\bm{b}={}^t(b_1,\dots,b_r)$. As mentioned above, the third term is $\delta_{a_1,a_2}w_1\tas w_2$.
\end{proof}

\begin{rem}\label{rem:MEShar}
By definition, MES of level $N$ also satisfy the harmonic product, i.e. $G:(\widetilde{\ha^2},\tas)\to\mathcal{O}(\HH)$ is an algebra homomorphism.
\end{rem}

Yuan--Zhao \cite{YZ16} defined regularized MZVs of level $N$ as polynomials in $T$ by using regularized MLVs and \Cref{prop:zetaL}. In this paper, we define the regularized MZVs as the constant terms of the Yuan--Zhao's regularized MZVs.

\begin{dfn}[Yuan--Zhao \cite{YZ16}]
Define the regularized multiple zeta values of level $N$ as the images of $\zeta^{\tsha}:\ha^1\to\CC$ and $\zeta^{\tas}:\ha^1\to\CC$, by
\begin{gather}
\zeta^{\tsha}(w)\coloneq(L^\mathrm{reg}_\sha\circ\pi\circ\rho)(w)|_{T=0},\\
\zeta^{\tas}(w)\coloneq(L^\mathrm{reg}_\ast\circ\pi)(w)|_{T=0}.
\end{gather}
\end{dfn}

By definition, regularized MZV of level $N$, $\zeta^{\tsha}$ and $\zeta^{\tas}$ satisfy the $\tsha$ and $\tas$-product, respectively.

Now, we give the antipode relations for MLVs and MZVs of level $N$. We can extend $L_\sha$ on $\ha$ with $\sha$-homomorphy by putting $L_\sha(x)=L_\sha(y_0)=0$. Hoffman (\cite{Hoff00}) provided a Hopf algebra structure on the quasi-shuffle algebra.

\begin{thm}[Hoffman {\cite[Theorem 3.2]{Hoff00}}]
$(\ha_\sha,\Delta_H,\eps_H,\mathcal{S})$ is a Hopf algebra with
\begin{align}
\Delta_H(w)=\sum_{uv=w}u\otimes v,\quad\eps_H(w)=\begin{cases}1&w=1\\0&w\ne1\end{cases},\quad\mathcal{S}(w)=(-1)^{\wt(w)}\overset{\leftarrow}{w},
\end{align}
where $\overset{\leftarrow}{w}=a_n\cdots a_1$ for $w=a_1\cdots a_n$.
\end{thm}

It is known that MLVs of shuffle type satisfy the antipode relations.

\begin{prop}\label{prop:antiL}
For $n_1,\dots,n_r\ge1$ and $a_1,\dots,a_{r-1}\in\ZZ/N\ZZ$, we have
$$\sum_{q=1}^r\sum_{\substack{k_1+\cdots+k_r=n\\k_q=1}}(-1)^{m_q}\prod_{\substack{i=1\\i\ne q}}^r\binom{k_i-1}{n_i-1}L_\sha\binom{k_{q-1},\dots,k_1}{a_{q-1},\dots,a_1}L_\sha\binom{k_{q+1},\dots,k_r}{a_q,\dots,a_{r-1}}=0$$
where $m_q=k_1+\cdots+k_{q-1}+n_q$ and $n=n_1+\cdots+n_r$.
\end{prop}

\begin{proof}
Considering the convolution of $\mathcal{S}$ and $1$, we have
$$\sum_{uv=w}(-1)^{\wt(u)}\overset{\leftarrow}{u}\sha v=m\circ(\mathcal{S}\otimes1)\circ\Delta_H(w)=(u\circ\eps_H)(w)=\begin{cases}1&w=1,\\0&w\ne1.\end{cases}$$
for any word $w$.
Then, by taking $w=x^{n_1-1}y_{a_1}x^{n_2-1}y_{a_2}\cdots x^{n_{r-1}-1}y_{a_{r-1}}x^{n_r-1}$, we have
\begin{align}
\sum_{q=1}^r\sum_{l_q=0}^{n_q-1}(-1)^{n_1+\cdots+n_{q-1}+l_q}&(x^{l_q}y_{a_{q-1}}x^{n_{q-1}-1}\cdots y_{a_1}x^{n_1-1}\\
&\sha x^{n_q-l_q-1}y_{a_q}x^{n_{q+1}-1}\cdots y_{a_{r-1}}x^{n_r-1})=0.
\end{align}
Applying to both sides the map $L_\sha$, we obtain the claim since one can show that for any $l\ge1$
$$L_\sha(x^ly_{a_1}x^{n1-1}\cdots y_{a_r}x^{n_r-1})=(-1)^l\sum_{k_1+\cdots+k_r=n_1+\cdots+n_r+l}\prod_{i=1}^r\binom{k_i-1}{n_i-1}L_\sha(k_1,\dots,k_r)$$
by induction and the definition of the shuffle product.
\end{proof}

Since MZV of level $N$ can be written via MLV, we obtain the antipode relations also for MZV of level $N$.

\begin{cor}[Antipode relation for MZV of level $N$]\label{cor:antizeta}
For $n_1,\dots,n_r\ge2$ and $a_1,\dots,a_r\in\ZZ/N\ZZ$, we have
$$\sum_{q=1}^r\sum_{\substack{k_1+\dots+k_r=n\\k_q=1}}(-1)^{m_q}\prod_{\substack{i=1\\i\ne q}}^r\binom{k_i-1}{n_i-1}\zeta\bigg(\begin{array}{@{}c@{}c@{}c@{}}k_{q-1}&,\dots,&k_1\\a_q-a_{q-1}&,\dots,&a_q-a_1\end{array}\bigg)\zeta\bigg(\begin{array}{@{}c@{}c@{}c@{}}k_{q+1}&,\dots,&k_r\\a_{q+1}-a_q&,\dots,&a_r-a_q\end{array}\bigg)=0,$$
where $m_q=k_1+\cdots+k_{q-1}+n_q$ and $n=n_1+\cdots+n_r$.
\end{cor}

\begin{proof}
Using \Cref{prop:zetaL} and \Cref{prop:antiL}, we have
\begin{align}
(\text{L.H.S})=&\sum_{b_1,\dots,b_{r-1}\in\ZZ/N\ZZ}\eta^{(\bm{a},\bm{b})}\sum_{q=1}^r\sum_{\substack{k_1+\cdots+k_r=n\\k_q=1}}(-1)^{k_1+\cdots+k_{q-1}+n_q}\\
&\times\prod_{\substack{i=1\\i\ne q}}^r\binom{k_i-1}{n_i-1}L_\sha\binom{k_{q-1},\dots,k_1}{b_{q-1},\dots,b_1}L_\sha\binom{k_{q+1},\dots,k_r}{b_q,\dots,b_{r-1}}=0,
\end{align}
where $(\bm{a},\bm{b})=-\sum_{i=1}^{r-1}(a_{i+1}-a_i)b_i$.
\end{proof}

\section{Fourier expansion for multiple Eisenstein series of level $N$}\label{sec:MES}

The Fourier expansion of MES of level $1$ is obtained by Bachmann\footnote{ H. Bachmann, \textit{Multiple Zeta–Werte und die Verbindung zu Modulformen durch
Multiple Eisensteinreihen}, Master’s thesis, Universit\"at Hamburg, 2012.} and written in Bachmann--Tasaka \cite{BT}. In this section, we give the Fourier expansion of MES of level $N$ explicitly, in the same way as Bachmann--Tasaka \cite{BT}. By the definition of the order of lattice points, we can split the sum $\sum_{0\prec\lambda_1\prec\cdots\prec\lambda_r}$ into $2^r$ many terms. In this section, we consider the each term and give its Fourier expansion. Let $\{\xx,\yy\}^\ast$ be the set of all words generated by letters $\xx$ and $\yy$.

\begin{dfn}
For $n_1,\dots,n_r\ge2$, $a_1,\dots,a_r\in{\ZZ/N\ZZ}$ and $w_1\cdots w_r\in\{\xx,\yy\}^\ast$, we define
\begin{align}
G_{w_1\cdots w_r}\bigg(\begin{matrix}n_1,\dots,n_r\\a_1,\dots,a_r\end{matrix};\tau\bigg)\coloneq\lim_{L\to\infty}\lim_{M\to\infty}\sum_{\substack{\lambda_i-\lambda_{i-1}\in P_{w_i}\\\lambda_i\in N\ZZ_L\tau+\ZZ_M,\lambda_i\equiv a_i\hspace{-8pt}\pmod{N}}}\frac{1}{\lambda_1^{n_1}\dots\lambda_r^{n_r}},
\end{align}
where $P_{\xx}=\{l\tau+m\in\ZZ\tau+\ZZ\mid l=0,m>0\}$, $P_{\yy}=\{l\tau+m\in\ZZ\tau+\ZZ\mid l>0\}$ and $l\tau+m\equiv a$ means $m\equiv a$.
\end{dfn}

Note that $\lambda\in P_{\xx}\sqcup P_{\yy}$ if and only if $\lambda\succ0$. 

\begin{lem}[cf. Bachmann--Tasaka {\cite[Proposition 2.2]{BT}}]
For $n_1,\dots,n_r\ge2$ and $a_1,\dots,a_r\in\ZZ/N\ZZ$, we have
\begin{align}
G\bigg(\begin{matrix}n_1,\dots,n_r\\a_1,\dots,a_r\end{matrix};\tau\bigg)=\sum_{w_1,\dots,w_r\in\{\xx,\yy\}}G_{w_1\cdots w_r}\bigg(\begin{matrix}n_1,\dots,n_r\\a_1,\dots,a_r\end{matrix};\tau\bigg).
\end{align}
\end{lem}

\subsection{Multitangent function of level $N$}

Multitangent functions are defined by Bouillot \cite{Bouillot}, and he studied the algebraic structure of multitangent functions. In this subsection, we define multitangent function of level $N$ and give its Fourier expansion.
\begin{dfn}
We define the multitangent function of level $N$ $\Psi:\ha^2\to\mathcal{O}(\HH)$ by
$$\Psi\bigg(\begin{matrix}n_1,\dots,n_r\\a_1,\dots,a_r\end{matrix};\tau\bigg)\coloneq\sum_{\substack{-\infty<m_1<\dots<m_r<+\infty\\m_i\equiv a_i\hspace{-8pt}\pmod{N}}}\frac{1}{(\tau+m_1)^{n_1}\dots(\tau+m_r)^{n_r}}$$
for $n_1,\dots,n_r\ge2$ and $a_1,\dots,a_r\in\ZZ/N\ZZ$, and $\Psi(\emptyset;\tau)=1$, together with $\QQ$-linearity. We define
$$\Psi\bigg(\begin{matrix}1\\a\end{matrix}\:;\tau\bigg)\coloneq\lim_{M\to\infty}\sum_{\substack{|m|<M\\m\equiv a\hspace{-8pt}\pmod{N}}}\frac{1}{\tau+m}.$$
for $a\in\ZZ/N\ZZ$ and $\tau\in\HH$.
\end{dfn}

Bouillot (\cite{Bouillot}) proved that any multitangent fuction can be written as a $\QQ$-linear sum of products of MZVs and monotangent functions. He actually treated $\ast$-regularized colored multitangent function. The following lemma can be obtained as a corollary of his result (\cite[Theorem 6]{Bouillot}) in the convergent case.

\begin{lem}\label{lem:multitan}
For $n_1,\dots,n_r\ge2$ and $a_1,\dots,a_r\in\ZZ/N\ZZ$, we have
\begin{align}
\Psi\bigg(\begin{matrix}n_1,\dots,n_r\\a_1,\dots,a_r\end{matrix};\tau\bigg)&=\sum_{q=1}^r\sum_{\substack{k_1+\dots+k_r=n\\k_i\ge1}}(-1)^{n+n_q+k_{q+1}+\dots+k_r}\prod_{\substack{p=1\\p\ne q}}^r\binom{k_p-1}{n_p-1}\\
&\times\zeta\bigg(\begin{array}{@{}c@{}c@{}c@{}}k_{q-1}&,\dots,&k_1\\a_q-a_{q-1}&,\dots,&a_q-a_1\end{array}\bigg)\zeta\bigg(\begin{array}{@{}c@{}c@{}c@{}}k_{q+1}&,\dots,&k_r\\a_{q+1}-a_q&,\dots,&a_r-a_q\end{array}\bigg)\Psi\bigg(\begin{matrix}k_q\\a_q\end{matrix}\:;\tau\bigg),
\end{align}
where $n=n_1+\cdots+n_r$.
\end{lem}

The following lemma gives us the Fourier expansion for multitangent functions of level $N$.

\begin{lem}[Yuan--Zhao {\cite[Lemma 4.1]{YZ15}}]\label{lem:monotan}
For an integer $n\ge1$ and $a\in\ZZ/N\ZZ$, we have
$$\Psi\bigg(\begin{matrix}n\\a\end{matrix}\:;N\tau\bigg)=\left(\frac{-2\pi\sqrt{-1}}{N}\right)^n\sum_{c>0}\frac{c^{n-1}\eta^{ac}}{(n-1)!}q^{c}-\delta_{n,1}\frac{\pi\sqrt{-1}}{N}.$$
\end{lem}

\subsection{Multiple divisor function of level $N$}

Multiple divisor functions are initially studied by Bachmann and K\"{u}hn (\cite{BK}). Yuan and Zhao (\cite{YZ16}) generalized it to arbitrary level and studied the relation to MZV of level $N$.

\begin{dfn}[Yuan--Zhao \cite{YZ16}]\label{dfn:div}
For $n_1,\dots,n_r\ge1$ and $a_1,\dots,a_r\in{\ZZ/N\ZZ}$, we define multiple divisor function of level $N$ $g:\ha^1\to\CC\llbracket q\rrbracket$ by
$$g\bigg(\begin{matrix}n_1,\dots,n_r\\a_1,\dots,a_r\end{matrix};q\bigg)\coloneq\left(\frac{-2\pi\sqrt{-1}}{N}\right)^n\sum_{\substack{0<d_1<\cdots<d_r\\c_1,\dots,c_r>0}}\prod_{i=1}^r\frac{\eta^{a_ic_i}c_i^{n_i-1}}{(n_i-1)!}q^{c_id_i}$$
for $n_1,\dots,n_r\ge1$ and $a_1,\dots,a_r\in\ZZ/N\ZZ$, and $g(\emptyset;q)=1$, together with $\QQ$-linearity, where $n=n_1+\cdots+n_r$.
\end{dfn}

As a holomorphic function on $\HH$, multiple divisor function $g(q)$ can be written as sum of products of monotangent function. The following lemma follows from \Cref{lem:monotan}.

\begin{lem}\label{lem:divisor}
For any $n_1,\dots,n_r\in\ZZ_{\ge2}$, $a_1,\dots,a_r\in{\ZZ/N\ZZ}$ and $\tau\in\HH$, we have
\begin{align}
g\bigg(\begin{matrix}n_1,\dots,n_r\\a_1,\dots,a_r\end{matrix};q\bigg)
=\sum_{0<d_1<\dots<d_r}\Psi\bigg(\begin{matrix}n_1\\a_1\end{matrix}\:;d_1N\tau\bigg)\cdots\Psi\bigg(\begin{matrix}n_r\\a_r\end{matrix}\:;d_rN\tau\bigg).
\end{align}
\end{lem}

\subsection{The Fourier expansion of MES of level $N$}

The Fourier expansion of MES can be written with MZVs and multiple divisor functions.

\begin{prop}\label{prop:Eisenstein}
For any $n_1,\dots,n_r\ge2$, $a_1,\dots,a_r\in{\ZZ/N\ZZ}$ and $w_1\cdots w_r\in\{\xx,\yy\}^*$, we put $w_1\cdots w_r=\xx^{t_1-1}\yy\xx^{t_2-t_1-1}\yy\cdots\xx^{t_h-t_{h-1}-1}\yy\xx^{r-t_h}
$. Then we have
\begin{align}
&G_{w_1\cdots w_r}\bigg(\begin{matrix}n_1,\dots,n_r\\a_1,\dots,a_r\end{matrix};\tau\bigg)\\
&=\zeta\binom{n_1,\dots,n_{t_1-1}}{a_1,\dots,a_{t_1-1}}\rising{5}{$\sum_{\substack{t_1\le q_1\le t_2-1\\[-5pt]\vdots\\t_h\le q_h\le r}}$}\rising{5}{$\sum_{\substack{k_{t_j}+\cdots+k_{t_{j+1}-1}\\=n_{t_j}+\cdots+n_{t_{j+1}-1}\\(1\le j\le h),k_i\ge1}}$}\prod_{j=1}^h\Bigg\{(-1)^{l_j}\Bigg(\prod_{\substack{p=t_j\\p\ne q_j}}^{t_{j+1}-1}\binom{k_p-1}{n_p-1}\Bigg)\\
&\times\zeta\bigg(\begin{array}{@{}c@{}c@{}c@{}}k_{q_j-1}&,\dots,&k_{t_j}\\a_{q_j}-a_{q_j-1}&,\dots,&a_{q_j}-a_{t_j}\end{array}\bigg)\zeta\bigg(\begin{array}{@{}c@{}c@{}c@{}}k_{q_j+1}&,\dots,&k_{t_{j+1}-1}\\a_{q_j+1}-a_{q_j}&,\dots,&a_{t_{j+1}-1}-a_{q_j}\end{array}\bigg)\Bigg\}
g\bigg(\begin{matrix}k_{q_1},\dots,k_{q_h}\\a_{q_1},\dots,a_{q_h}\end{matrix};q\bigg),
\end{align}
where $l_j=n_{t_j}+\cdots+n_{t_{j+1}-1}+n_{q_j}+k_{q_j+1}+\cdots+k_{q_{j+1}-1}$, $t_{h+1}=r+1$ and $a_{r+1}=0$.
\end{prop}

\begin{proof}
By definition of the portion $G_{w_1\cdots w_r}$, we have
\begin{align}
&G_{w_1\cdots w_r}\bigg(\begin{matrix}n_1,\dots,n_r\\a_1,\dots,a_r\end{matrix};\tau\bigg)\\
&=\sum_{\substack{0<l_1<\cdots<l_h\\0\le{}^\forall j\le h,m_{t_j}<\cdots<m_{t_{j+1}-1}\\{}^\forall i>0,m_i>0,m_i\equiv a_i\hspace{-8pt}\pmod{N}}}\frac{1}{m_1^{n_1}\cdots m_{t_1-1}^{n_{t_1-1}}}\prod_{j=1}^h\frac{1}{(l_jN\tau+m_{t_j})^{n_{t_j}}\cdots(l_jN\tau+m_{t_{j+1}-1})^{n_{t_{j+1}-1}}}\\
&=\zeta\binom{n_1,\dots,n_{t_1-1}}{a_1,\dots,a_{t_1-1}}\sum_{0<l_1<\cdots<l_h}\prod_{j=1}^h\Psi\bigg(\begin{matrix}n_{t_j},\dots,n_{t_{j+1}-1}\\a_{t_j},\dots,a_{t_{j+1}-1}\end{matrix};l_jN\tau\bigg).
\end{align}
By \Cref{lem:multitan}, we have
\begin{align}
&\sum_{0<l_1<\cdots<l_h}\prod_{j=1}^h\Psi\bigg(\begin{matrix}n_{t_j},\dots,n_{t_{j+1}-1}\\a_{t_j},\dots,a_{t_{j+1}-1}\end{matrix};l_jN\tau\bigg)\\
&=\lim_{L\to\infty}\rising{5}{$\sum_{\substack{t_1\le q_1\le t_2-1\\[-5pt]\vdots\\t_h\le q_h\le r}}$}\,\rising{5}{$\sum_{\substack{k_{t_j}+\cdots+k_{t_{j+1}-1}\\=n_{t_j}+\cdots+n_{t_{j+1}-1}\\(1\le j\le h),k_i\ge1}}$}\prod_{j=1}^h\Bigg\{(-1)^{l_j}\prod_{\substack{p=t_j\\p\ne q_j}}^{t_{j+1}-1}\binom{k_p-1}{n_p-1}\\
&\times\zeta\bigg(\begin{array}{@{}c@{}c@{}c@{}}k_{q_j-1}&,\dots,&k_{t_j}\\a_{q_j}-a_{q_j-1}&,\dots,&a_{q_j}-a_{t_j}\end{array}\bigg)\zeta\bigg(\begin{array}{@{}c@{}c@{}c@{}}k_{q_j+1}&,\dots,&k_{t_{j+1}-1}\\a_{q_j+1}-a_{q_j}&,\dots,&a_{t_{j+1}-1}-a_{q_j}\end{array}\bigg)\Bigg\}\\
&\times\sum_{0<l_1<\dots<l_h<L}\Psi\bigg(\begin{matrix}k_{q_1}\\a_{q_1}\end{matrix}\:;l_1N\tau\bigg)\cdots\Psi\bigg(\begin{matrix}k_{q_h}\\a_{q_h}\end{matrix}\:;l_hN\tau\bigg).
\end{align}
The constant term in the right hand side vanish if $k_{q_i}=1$ for some $i=1,\dots,h$ by \Cref{cor:antizeta}. The claim follows now from \Cref{lem:divisor}.
\end{proof}

\section{Goncharov coproduct for formal iterated integrals}\label{sec:FMZV}

In this section, we consider the algebra generated by formal iterated integrals, which is introduced by Goncharov \cite{Gon05} and calculate the coproduct for the formal iterated integrals corresponding to MZV of level $N$.

\subsection{Hopf algebra of formal iterated integrals}

Goncharov(\cite{Gon05}) considered the formal version of iterated integrals
\begin{align}\label{eq:formalI}
\int_{a_0}^{a_{m+1}}\frac{dt}{t-a_1}\cdots\frac{dt}{t-a_m}\quad(a_0,\dots,a_{m+1}\in\CC),
\end{align}
and proved the algebra generated by formal iterated integrals has a Hopf algebra structure.

\begin{dfn}[Goncharov \cite{Gon05}]\label{dfn:I(S)}
Let $S$ be a set. Define a commutative graded $\QQ$-algebra $\mathcal{I}(S)$ by
$$\mathcal{I}(S)\coloneq\QQ\left[\II(a_0;a_1,\dots,a_m;a_{m+1})\mid m\ge0,a_i\in S\right]/_{(\mathrm{i})\sim(\mathrm{iv})},$$
where $\mathrm{deg}(\II(a_0;a_1,\dots,a_m;a_{m+1}))=m$ and the quotient modulo the ideal generated by (i) $\sim$ (iv). The relations $(\mathrm{i})\sim(\mathrm{iv})$ are the following:
\begin{enumerate}[label=(\roman*)]
\item$\II(a;b)=1,\quad(a,b\in S).$ \label{(i)}
\item(Shuffle product formula) For $a,b,a_1,\dots,a_{n+m}\in S$, it holds
$$\II(a;a_1,\dots,a_n;b)\II(a;a_{n+1},\dots,a_{n+m};b)=\sum_{\sigma\in Sh^{(n+m)}_n}\II(a;a_{\sigma^{-1}(1)},\dots,a_{\sigma^{-1}(n+m)};b),$$
where
$$Sh^{(n+m)}_n\coloneq\left\{\sigma\in\mathfrak{S}_{n+m}\mid\sigma(1)<\cdots<\sigma(n),\sigma(n+1)<\cdots<\sigma(n+m)\right\}.$$\label{(ii)}
\item(Path composition formula) For $x,a_0,\dots,a_{m+1}\in S$, it holds
$$\II(a_0;a_1,\dots,a_m;a_{m+1})=\sum_{i=0}^m\II(a_0;a_1,\dots,a_i;x)\II(x;a_{i+1},\dots,a_m;a_{m+1}).$$\label{(iii)}
\item$\II(a;a_1,\dots,a_m;a)=0,\quad(a,a_1,\dots,a_m\in S,m\ge1).$\label{(iv)}
\end{enumerate}
\end{dfn}
\begin{thm}[Goncharov \cite{Gon05}]
$\mathcal{I}(S)$ is a graded Hopf algebra with the coproduct $\Delta:\mathcal{I}(S)\rightarrow\mathcal{I}(S)\otimes_{\QQ}\mathcal{I}(S)$ defined by
\begin{align}
&\Delta(\II(a_0;a_1,\dots,a_m;a_{m+1}))\\
\coloneq&\sum_{0=i_0<i_1<\cdots<i_{k}<i_{k+1}=m+1}\prod_{p=0}^k\II(a_{i_p};a_{i_{p+1}},\dots,a_{i_{p+1}-1};a_{i_{p+1}})\otimes\II(a_0;a_{i_1},\dots,a_{i_k};a_{m+1}).
\end{align}
\end{thm}

\begin{rem}
The counit $\eps_G:\mathcal{I}(S)\to\QQ$ is defined by
$$\eps_G(u)=\begin{cases}1&\mathrm{deg}(u)=0\\0&\mathrm{deg}(u)>0\end{cases}.$$
The antipode is determined inductively on the degree.
\end{rem}

\subsection{Formal iterated integrals corresponding MLV and MZV of level $N$}\label{subsection:FMZV}

Hereinafter, we consider the case $S=\{\eta,\eta^2,\dots,\eta^N,0\}$ and denote $\mathcal{I}=\mathcal{I}(S)$. Let $\mathfrak{a}$ be an ideal of $\mathcal{I}$ generated by $\{\II(0;0;a)\mid a\in S\setminus\{0\}\}$, and let $\mathcal{I}^0$ be the quotient $\mathcal{I}^0=\mathcal{I}/\mathfrak{a}$. The following proposition follows since an ideal generated by a primitive element is a Hopf ideal.

\begin{prop}[Bachmann--Tasaka \cite{BT}]
$(\mathcal{I}^0,\Delta)$ is a Hopf algebra.
\end{prop}

We give some important properties of formal iterated integrals.

\begin{lem}[Goncharov \cite{Gon05}]\label{lem:G1}
For any $a_0,\dots,a_{m+1}\in S$, we have
$$\II(a_0;a_1,\dots,a_m;a_{m+1})=(-1)^m\II(a_{m+1};a_m,\dots,a_1;a_0).$$
\end{lem}

For $n_1,\dots,n_r\ge1$ and $a,a_1,\dots,a_r\in\ZZ/N\ZZ$, denote
\begin{gather}
\II_a\binom{n_1,\dots,n_r}{a_1,\dots,a_r}\coloneq\II(0;\eta^{a_1},\{0\}^{n_1-1},\dots,\eta^{a_r},\{0\}^{n_r-1};\eta^a),\\
\II\binom{n_1,\dots,n_r}{a_1,\dots,a_r}\coloneq\II_0\binom{n_1,\dots,n_r}{a_1,\dots,a_r}.
\end{gather}
\Cref{dfn:L} implies $\II\binom{n_1,\dots,n_r}{a_1,\dots,a_r}$ corresponds to MLV of shuffle type
$$(-1)^rL_{\sha}\binom{n_1,\dots,n_r}{-a_1,\dots,-a_r}=\int_0^1\frac{dt}{t-\eta^{a_1}}\left(\frac{dt}{t}\right)^{n_1-1}\cdots\frac{dt}{t-\eta^{a_r}}\left(\frac{dt}{t}\right)^{n_r-1}.$$

\begin{lem}[see Brown \cite{Br12}, Bachmann--Tasaka \cite{BT}]\label{lem:G2}
For any $n,n_1,\dots,n_r\ge1$ and $a,a_1,\dots,a_r\in\ZZ/N\ZZ$, we have
\begin{align}
&\II(0;\{0\}^n,\eta^{a_1},\{0\}^{n_1-1},\dots,\eta^{a_r},\{0\}^{n_r-1};\eta^a)\\
&=(-1)^n\sum_{\substack{k_1+\cdots+k_r=n+n_1+\cdots+n_r\\k_i\ge n_i}}\prod_{p=1}^r\binom{k_p-1}{n_p-1}\II_a\binom{k_1,\dots,k_r}{a_1,\dots,a_r}.
\end{align}
\end{lem}

Using these properties, we know that any element of $\mathcal{I}^0$ can be written as a polynomial of $\II_a\binom{n_1,\dots,n_r}{a_1,\dots,a_r}$

\begin{prop}
It holds
$$\mathcal{I}^0=\QQ\left[\II_a\binom{n_1,\dots,n_r}{a_1,\dots,a_r}\relmid r\ge0,n_i\ge1,a,a_i\in\ZZ/N\ZZ\right].$$
\end{prop}

\begin{proof}
Any element $\II(a_0;a_1,\dots,a_m;a_{m+1})\in\mathcal{I}^0$ can be expressed as a sum of  products for some $\II(0;\dots;a)$ $(a\in S)$ by \ref{(iii)} and \Cref{lem:G1}. Using \ref{(ii)} and \Cref{lem:G2}, it can be written as a sum of products for some $\II_a\binom{n_1,\dots,n_r}{a_1,\dots,a_r}$. Now, the products can be expressed as a $\QQ$-linear combination of them by using the shuffle product formula.
\end{proof}

Let $\mathcal{I}^1$ be a subalgebra of $\mathcal{I}^0$ defined by
$$\mathcal{I}^1\coloneq\left\langle\II\binom{n_1,\dots,n_r}{a_1,\dots,a_r}\in\mathcal{I}^0\relmid r\ge0,n_i>0,a_i\in\ZZ/N\ZZ\right\rangle_\QQ$$
and let $\mu:\mathcal{I}^0\twoheadrightarrow\mathcal{I}^1$ be a surjective algebra homomorphism defined by
$$\mu\left(\II_a\binom{n_1,\dots,n_r}{a_1,\dots,a_r}\right)=\II\bigg(\begin{array}{@{}c@{}c@{}c@{}c}n_1&,\dots,&n_r\\a_1-a&,\dots,&a_r-a\end{array}\bigg).$$

\begin{rem}
It holds $$\int_0^{\eta^a}\frac{dt}{t-\eta^{a_1}}\left(\frac{dt}{t}\right)^{n_1-1}\!\!\!\!\!\!\cdots\frac{dt}{t-\eta^{a_r}}\left(\frac{dt}{t}\right)^{n_r-1}\!\!\!\!=\int_0^1\frac{dt}{t-\eta^{a_1-a}}\left(\frac{dt}{t}\right)^{n_1-1}\!\!\!\!\!\!\cdots\frac{dt}{t-\eta^{a_r-a}}\left(\frac{dt}{t}\right)^{n_r-1}$$ if the both sides conveges. $\mu$ is an operator corresponding to such variable changing of integrals.
\end{rem}

Let $\Delta_\mu:\mathcal{I}^1\to\mathcal{I}^1\otimes\mathcal{I}^1$ be an algebra homomorphism defined by
$$\Delta_\mu\coloneq(\mu\otimes\mu)\circ\Delta|_{\mathcal{I}^1}.$$

\begin{prop}
$(\mathcal{I}^1,\Delta_\mu,\eps_G)$ is a Hopf algebra.
\end{prop}

\begin{proof}
It is clear that $\eps_G$ and $\Delta_\mu$ satisfy the counitary property. Let us check the coassociativity. If it holds $\Delta_\mu\circ\mu=\Delta_\mu$ on $\mathcal{I}^0$, we have
$$(\Delta_\mu\otimes\mathrm{id})\circ\Delta_\mu(u)=(\Delta_\mu\otimes\mathrm{id})\sum\mu(u_1)\otimes\mu(u_2)=\sum\mu(u_1)\otimes\mu(u_2)\otimes\mu(u_3)=(\mathrm{id}\otimes\Delta_\mu)\circ\Delta_\mu(u).$$
So it suffices to show that
$$\Delta_\mu\circ\mu\left(\II_a\binom{n_1,\dots,n_r}{a_1,\dots,a_r}\right)=\Delta_\mu\left(\II_a\binom{n_1,\dots,n_r}{a_1,\dots,a_r}\right)$$
for any $n_1,\dots,n_r\ge1$ and $a,a_1,\dots,a_r\in\ZZ/N\ZZ$. This statement follows from the calculation of $\Delta$ in the next subsection (\cref{lem:phi}). The compatibility of the shuffle product and coproduct $\Delta$ follows from $\Delta_\mu\circ\mu=\Delta_\mu$ and the compatibility of product and $\mu$. The antipode is determined inductively since the product and coproduct preserve the degree.
\end{proof}

\subsection{Computing Goncharov coproduct}\label{subsection:CompGon}

In this subsection, we give the explicit formula for the Goncharov coproduct of $\II\binom{n_1,\dots,n_r}{a_1,\dots,a_r}$, which correspond to MLVs of shuffle type. Then, we consider formal iterated integrals corresponding MZVs of level $N$ and give its coproduct.

In the same way as Bachmann--Tasaka \cite{BT}, we calcurate $\Delta\left(\II\binom{n_1,\dots,n_r}{a_1,\dots,a_r}\right)$ by splitting into $2^r$ many terms.

For positive integers $0=i_0<i_1<\cdots<i_k<i_{k+1}=n+1$ $(0\le k\le n)$ and  $\eps_1,\dots,\eps_n\in S$, we define $\varphi_{i_1,\dots,i_k}(\eps_1,\dots,\eps_n)\in\mathcal{I}^0\otimes\mathcal{I}^0$ by
$$\varphi_{i_1,\dots,i_k}(\eps_1,\dots,\eps_n)\coloneq\prod_{p=0}^k\II(\eps_{i_p};\eps_{i_{p+1}},\dots,\eps_{i_{p+1}-1};\eps_{i_{p+1}})\otimes\II(0;\eps_{i_1},\dots,\eps_{i_k};1),$$
where $\eps_0=0$, $\eps_{n+1}=1$, and denote
$$\varphi_{i_1,\dots,i_k}\binom{n_1,\dots,n_r}{a_1,\dots,a_r}\coloneq\varphi_{i_1,\dots,i_k}(\eta^{a_1},\{0\}^{n_1-1},\dots,\eta^{a_r},\{0\}^{n_r-1}).$$
Further, we put
$$\iota_{n_1,\dots,n_r}(w_1\cdots w_r)\coloneq\{n_1+\cdots+n_{t_1-1}+1,\dots,n_1+\cdots+n_{t_h-1}+1\}$$
for $n_1,\dots,n_r\ge1$ and $w_1\cdots w_r=\xx^{t_1-1}\yy\xx^{t_2-t_1-1}\yy\cdots\xx^{t_h-t_{h-1}-1}\yy\xx^{r-t_h}\in\{\xx,\yy\}^*$ $(0<t_1<\cdots<t_n<r+1)$.

\begin{dfn}
For any word $w_1\cdots w_r=\xx^{t_1-1}\yy\xx^{t_2-t_1-1}\yy\cdots\xx^{t_h-t_{h-1}-1}\yy\xx^{r-t_h}\in\{\xx,\yy\}^*$, $n_1,\dots,n_r\ge1$ and $a_1,\dots,a_r\in\ZZ/N\ZZ$, we define
\begin{align}
\Phi_{w_1\cdots w_r}\binom{n_1,\dots,n_r}{a_1,\dots,a_r}\coloneq\sum_{k=h}^n\sum_{\substack{1\le i_1<\cdots<i_k\le n\\\{i_1,\dots,i_k\}\cap\{1,n_1+1,\dots,n_1+\cdots+n_{r-1}+1\}\\=\iota_{n_1,\dots,n_r}(w_1\cdots w_r)}}\varphi_{i_1,\dots,i_k}\binom{n_1,\dots,n_r}{a_1,\dots,a_r},
\end{align}
where $n=n_1+\cdots+n_r$.
\end{dfn}

\begin{lem}\label{lem:DeltaPhi}
For any $n_1,\dots,n_r\ge1$ and $a_1,\dots,a_r\in\ZZ/N\ZZ$, we have
$$\Delta\left(\II\binom{n_1,\dots,n_r}{a_1,\dots,a_r}\right)=\sum_{w_1,\dots,w_r\in\{\xx,\yy\}}\Phi_{w_1\cdots w_r}\binom{n_1,\dots,n_r}{a_1,\dots,a_r}.$$
\end{lem}

\begin{proof}
By definition of the coproduct $\Delta$, we have
$$\Delta\left(\II\binom{n_1,\dots,n_r}{a_1,\dots,a_r}\right)=\sum_{k=0}^n\sum_{1\le i_1<\cdots<i_k\le n}\varphi_{i_1,\dots,i_k}\binom{n_1,\dots,n_r}{a_1,\dots,a_r}.$$
Meanwhile, it holds
\begin{align}
\sum_{w_1,\dots,w_r\in\{\xx,\yy\}}\Phi_{w_1\cdots w_r}\binom{n_1,\dots,n_r}{a_1,\dots,a_r}&=\sum_{h=0}^r\sum_{\substack{w_1,\dots,w_r\in\{\xx,\yy\}\\\deg_{\yy}(w_1\cdots w_r)=h}}\Phi_{w_1\cdots w_r}\binom{n_1,\dots,n_r}{a_1,\dots,a_r}\\
&=\sum_{h=0}^r\sum_{k=h}^r\sum_{\substack{1\le i_1<\cdots<i_k\le n\\\#\{i_1,\dots,i_k\}\cap P_{n_1,\dots,n_r}=h}}\varphi_{i_1,\dots,i_k}\binom{n_1,\dots,n_r}{a_1,\dots,a_r}\\
&=\sum_{k=0}^n\sum_{1\le i_1<\cdots<i_k\le n}\varphi_{i_1,\dots,i_k}\binom{n_1,\dots,n_r}{a_1,\dots,a_r},
\end{align}
where $P_{n_1,\dots,n_r}=\{1,n_1+1,n_1+n_2+1,\dots,n_1+\cdots+n_{r-1}+1\}$.
\end{proof}

The following lemma gives us the explicit formula for the Goncharov coproduct of $\II\binom{n_1,\dots,n_r}{a_1,\dots,a_r}$.

\begin{lem}\label{lem:phi}
For any $n_1,\dots,n_r\ge1$, $a_1,\dots,a_r\in\ZZ/N\ZZ$ and any word $w_1\cdots w_r=\xx^{t_1-1}\yy\xx^{t_2-t_1-1}\yy$ $\cdots\xx^{t_h-t_{h-1}-1}\yy\xx^{r-t_h}\in\{\xx,\yy\}^*$, we have
\begin{align}
&\Phi_{w_1\cdots w_r}\binom{n_1,\dots,n_r}{a_1,\dots,a_r}=\left(\II_{a_{t_1}}\binom{n_1,\dots,n_{t_1-1}}{a_1,\dots,a_{t_1-1}}\otimes1\right)\\
&\times\rising{5}{$\sum_{\substack{t_1\le q_1\le t_2-1\\[-5pt]\vdots\\t_h\le q_h\le r}}$}\rising{5}{$\sum_{\substack{k_{t_j}+\cdots+k_{t_{j+1}-1}\\=n_{t_j}+\cdots+n_{t_{j+1}-1}\\(1\le j\le h),k_i\ge1}}$}\prod_{j=1}^h\Bigg\{(-1)^{l_j}\Bigg(\prod_{\substack{p=t_j\\p\ne q_j}}^{t_{j+1}-1}\binom{k_p-1}{n_q-1}\Bigg)\\
&\times\II_{a_{t_j}}\binom{k_{q_j-1},\dots,k_{t_j}}{a_{q_j},\dots,a_{t_j+1}}\II_{a_{t_{j+1}}}\binom{k_{q_j+1},\dots,k_{t_{j+1}-1}}{a_{q_j+1},\dots,a_{t_{j+1}-1}}\Bigg\}\otimes\II\binom{k_{q_1},\dots,k_{q_h}}{a_{t_1},\dots,a_{t_h}},
\end{align}
where $l_j=n_{t_j}+\cdots+n_{t_{j+1}-1}+n_{q_j}+k_{q_j+1}+\cdots+k_{q_{j+1}-1}$, $t_{h+1}=r+1$ and $a_{r+1}=0$.
\end{lem}

\begin{proof}
The left-hand side is a sum of all terms of $\Delta\left(\II\binom{n_1,\dots,n_r}{a_1,\dots,a_r}\right)$ such that the edges in the diagram are $0,\eta^{a_{t_1}},\dots,\eta^{a_{t_h}}$. Using the path composition formula
\begin{align}
&\II(\eta^{a_{t_j}};\{0\}^{n_{t_j-1}},\eta^{a_{t_j+1}},\dots,\{0\}^{n_{t_{j+1}-1}-1};\eta^{a_{t_{j+1}}})\\
&=\sum_{t_j\le q_j\le t_{j+1}-1}\sum_{0\le l_{q_j}\le n_{q_j}-1}\II(\eta^{a_{t_j}};\{0\}^{n_{t_j}-1},\eta^{a_{t_j+1}},\dots,\eta^{a_{q_j}},\{0\}^{l_{q_j}};0)\\
&\times\II(0;\{0\}^{n_{q_j}-l_{q_j}-1},\eta^{a_{q_j+1}},\dots,\{0\}^{n_{t_{j+1}-1}-1};\eta^{a_{t_{j+1}}}),
\end{align}
we have
\begin{align}
&\Phi_{w_1\cdots w_r}\binom{n_1,\dots,n_r}{a_1,\dots,a_r}\\
&=(\II(0;\eta^{a_1},\{0\}^{n_1-1},\dots,\eta^{a_{t_1-1}},\{0\}^{n_{t_1-1}};\eta^{a_{t_1}})\otimes1)\\
&\times\bigg(\rising{8}{$\sum_{\substack{t_1\le q_1\le t_2-1\\[-5pt]\vdots\\t_h\le q_h\le r}}\sum_{\substack{0<l_{q_1}+k_{q_1}\le n_{q_1}\\[-7pt]\vdots\\[1pt]0<l_{q_h}+k_{q_h}\le n_{q_h}\\l_i\ge0,k_i\ge1}}$}\prod_{j=1}^h\II(\eta^{a_{t_j}};\{0\}^{n_{t_j}-1},\eta^{a_{t_j+1}},\{0\}^{n_{t_j+1}-1},\dots,\eta^{a_{q_j}},\{0\}^{l_{q_j}};0)\\
&\times(\II(0;0))^{k_{q_j}-2}\II(0;\{0\}^{n_{q_j}-l_{q_j}-k_{q_j}},\eta^{a_{q_j+1}},\{0\}^{n_{q_j+1}-1},\dots,\eta^{a_{t_{j+1}-1}},\{0\}^{n_{t_{j+1}-1}-1};\eta^{a_{t_{j+1}}})\\
&\otimes\II(0;\eta^{a_{t_1}},\{0\}^{k_{q_1}-1},\dots,\eta^{a_{t_h}},\{0\}^{k_{q_h}-1};1)\bigg).
\end{align}
Here, when $k_{q_j}=1$, we understand $(I(0;0))^{k_{q_j}-2}=1$. By \Cref{lem:G1},\Cref{lem:G2}, we have
\begin{align}
&\II(\eta^{a_{t_j}};\{0\}^{n_{t_j}-1},\eta^{a_{t_j+1}},\{0\}^{n_{t_j+1}-1},\dots,\eta^{a_{q_j}},\{0\}^{l_{q_j}};0)\\
&=(-1)^{n_{t_j}+\cdots+n_{q_j-1}+l_{q_j}}\II(0;\{0\}^{l_{q_j}},\eta^{a_{q_j}},\dots,\{0\}^{n_{t_j+1}-1},\eta^{a_{t_j+1}},\{0\}^{n_{t_j}-1};\eta^{a_{t_j}})\\
&=(-1)^{n_{t_j}+\cdots+n_{q_j-1}}\hspace{-20pt}\sum_{\substack{k_{t_j}+\cdots+k_{q_j-1}\\=n_{t_j}+\cdots+n_{q_j-1}+l_{q_j}}}\hspace{-5pt}\prod_{p=t_j}^{q_j-1}\binom{k_p-1}{n_p-1}\II_{a_{t_j}}\binom{k_{q_j-1},\dots,k_{t_j}}{a_{q_j},\dots,a_{t_{j+1}}},
\end{align}
and
\begin{align}
&\II(0;\{0\}^{n_{q_j}-l_{q_j}-k_{q_j}},\eta^{a_{q_j+1}},\{0\}^{n_{q_j+1}-1},\dots,\eta^{a_{t_{j+1}-1}},\{0\}^{n_{t_{j+1}-1}-1};\eta^{a_{t_{j+1}}})\\
&=(-1)^{n_{q_j}-l_{q_j}-k_{q_j}}\hspace{-40pt}\sum_{\substack{k_{q_j+1}+\cdots+k_{t_{j+1}-1}\\=n_{q_j}-l_{q_j}-k_{q_j}+n_{q_j+1}+\cdots+n_{t_{j+1}-1}}}\hspace{-20pt}\prod_{p=q_j+1}^{t_{j+1}-1}\binom{k_p-1}{n_p-1}\II_{a_{t_{j+1}}}\binom{k_{q_j+1},\dots,k_{t_{j+1}-1}}{a_{q_j+1},\dots,a_{t_{j+1}-1}}.
\end{align}
Therefore, we have
\begin{align}
&\Phi_{w_1\cdots w_r}\binom{n_1,\dots,n_r}{a_1,\dots,a_r}=\left(\II_{a_{t_1}}\binom{n_1,\dots,n_{t_1-1}}{a_1,\dots,a_{t_1-1}}\otimes1\right)\\
&\times\sum_{\substack{t_1\le q_1\le t_2-1\\[-5pt]\vdots\\t_h\le q_h\le r}}\sum_{(k_{q_1},l_{q_1},\dots,k_{q_h},l_{q_h})}\prod_{j=1}^h\Bigg\{(-1)^{n_{t_j}+\cdots+n_{q_j}-l_{q_j}-k_{q_j}}\Bigg(\prod_{\substack{p=t_j\\p\ne q_j}}^{t_{j+1}-1}\binom{k_p-1}{n_q-1}\Bigg)\\
&\times\II_{a_{t_j}}\binom{k_{q_j-1},\dots,k_{t_j}}{a_{q_j},\dots,a_{t_j+1}}\II_{a_{t_{j+1}}}\binom{k_{q_j+1},\dots,k_{t_{j+1}-1}}{a_{q_j+1},\dots,a_{t_{j+1}-1}}\Bigg\}\otimes\II\binom{k_{q_1},\dots,k_{q_h}}{a_{t_1},\dots,a_{t_h}}.
\end{align}
Here, the second sum runs over
\begin{align}
\left\{(k_{q_1},l_{q_1},\dots,k_{q_h},l_{q_h})\relmid\begin{array}{l}l_i\ge0,k_i\ge1,0<l_{q_j}+k_{q_j}\le n_{q_j},\\k_{t_j}+\cdots+k_{q_j-1}=n_{t_j}+\cdots+n_{q_j-1}+l_{q_j},\\k_{q_j+1}+\cdots+k_{t_{j+1}-1}=n_{q_j+1}+\cdots+n_{t_{j+1}-1}+n_{q_j}-l_{q_j}-k_{q_j}.\end{array}\right\}.
\end{align}
This is exactly the right-hand side of the claim.
\end{proof}

Let $\IT$ be a formal iterated integral corresponding to MZV of level $N$ defined by
$$\IT\binom{n_1,\dots,n_r}{a_1,\dots,a_r}\coloneq\frac{(-1)^r}{N^r}\sum_{b_1,\dots,b_r\in\ZZ/N\ZZ}\eta^{\rho(\bm{a})\cdot\bm{b}}\II\binom{n_1,\dots,n_r}{b_1,\dots,b_r}\in\mathcal{I}^1\otimes_\QQ\QQ(\eta)$$
for $n_1,\dots,n_r\ge1$ and $a_1,\dots,a_r\in\ZZ/N\ZZ$, where $\rho(\bm{a})=(a_1,a_2-a_1,\dots,a_r-a_{r-1})$.
\Cref{prop:zetaL} implies that $\IT$ corresponds to $\zeta$. Let
$$\widetilde{\Phi}_{w_1\cdots w_r}\binom{n_1,\dots,n_r}{a_1,\dots,a_r}\coloneq(\mu\otimes\mu)\left(\frac{(-1)^r}{N^r}\sum_{b_1,\dots,b_r\in\ZZ/N\ZZ}\eta^{\rho(\bm{a})\cdot\bm{b}}\Phi_{w_1\cdots w_r}\binom{n_1,\dots,n_r}{b_1,\dots,b_r}\right).$$
By \Cref{lem:DeltaPhi}, we have
$$\Delta_\mu\left(\IT\binom{n_1,\dots,n_r}{a_1,\dots,a_r}\right)=\sum_{w_1,\dots,w_r\in\{\xx,\yy\}}\widetilde{\Phi}_{w_1\cdots w_r}\binom{n_1,\dots,n_r}{a_1,\dots,a_r}.$$
The following proposition gives us the explicit formula for the Goncharov coproduct of formal iterated integrals corresponding to MZVs of level $N$.

\begin{prop}\label{prop:phitilde}
For any $n_1,\dots,n_r\ge1$, $a_1,\dots,a_r\in\ZZ/N\ZZ$ and $w_1,\dots,w_r\in\{\xx,\yy\}$, we have
\begin{align}
&\widetilde{\Phi}_{w_1\cdots w_r}\binom{n_1,\dots,n_r}{a_1,\dots,a_r}=\left(\IT\binom{n_1,\dots,n_{t_1-1}}{a_1,\dots,a_{t_1-1}}\otimes1\right)\\
&\times\rising{5}{$\sum_{\substack{t_1\le q_1\le t_2-1\\[-5pt]\vdots\\t_h\le q_h\le r}}$}\rising{5}{$\sum_{\substack{k_{t_j}+\cdots+k_{t_{j+1}-1}\\=n_{t_j}+\cdots+n_{t_{j+1}-1}\\(1\le j\le h),k_i\ge1}}$}\prod_{j=1}^h\Bigg\{(-1)^{l_j}\Bigg(\prod_{\substack{p=t_j\\p\ne q_j}}^{t_{j+1}-1}\binom{k_p-1}{n_p-1}\Bigg)\\
&\times\IT\bigg(\begin{array}{@{}c@{}c@{}c@{}}k_{q_j-1}&,\dots,&k_{t_j}\\a_{q_j}-a_{q_j-1}&,\dots,&a_{q_j}-a_{t_j}\end{array}\bigg)\IT\bigg(\begin{array}{@{}c@{}c@{}c@{}}k_{q_j+1}&,\dots,&k_{t_{j+1}-1}\\a_{q_j+1}-a_{q_j}&,\dots,&a_{t_{j+1}-1}-a_{q_j}\end{array}\bigg)\Bigg\}\otimes\IT\binom{k_{q_1},\dots,k_{q_h}}{a_{q_1},\dots,a_{q_h}},
\end{align}
where $l_j=n_{t_j}+\cdots+n_{t_{j+1}-1}+n_{q_j}+k_{q_j+1}+\cdots+k_{q_{j+1}-1}$, $t_{h+1}=r+1$ and $a_{r+1}=0$.
\end{prop}

\begin{proof}
It follows from \Cref{lem:phi} and variable substitution.
\end{proof}

We define the $\QQ$-algebra generated by formal iterated integrals corresponding to MZVs $\widetilde{\mathcal{I}^1}$ by
$$\widetilde{\mathcal{I}^1}\coloneq\left\langle\IT\binom{n_1,\dots,n_r}{a_1,\dots,a_r}\relmid r\ge0,n_1,\dots,n_r\ge1,a_1,\dots,a_r\in\ZZ/N\ZZ\right\rangle_\QQ.$$
$\widetilde{\mathcal{I}^1}$ is isomorphic to $(\ha^1,\tsha)$ as a $\QQ$-algebra by sending $\IT\binom{n_1,\dots,n_r}{a_1,\dots,a_r}$ to $z_{n_1,a_1}\cdots z_{n_r,a_r}$. In other words, we can equip $(\ha^1,\tsha)$ with a Hopf algebra structure via this identification.

Comparing \Cref{prop:Eisenstein} and \Cref{prop:phitilde}, gives one of our main results.

\newtheorem*{main1}{\Cref{thm:main1}}
\begin{main1}
For any $w\in\ha^2$, we have
$$G(w;\tau)=(\zeta\star g)(w;q).$$
\end{main1}

\section{Shuffle regularization for MES of level $N$}\label{sec:sha}

In this section, we construct the shuffle regularized MES of level $N$ by using the $\tsha$-homomorphism $g^{\tsha}$. 

\subsection{Shuffle regularization for multiple divisor function}\label{subsection:regg}

Kitada constructed shuffle regularized multiple divisor function in his master thesis \cite{Kitada23}. The thesis is not available and written in Japanese, so we introduce his results in this subsection.

\begin{dfn}[Kitada \cite{Kitada23}]
For $n_1,\dots,n_r\ge1$ and $a_1,\dots,a_r\in\ZZ/N\ZZ$, we define $H,g\in\CC\llbracket q\rrbracket\llbracket x_1,\dots,x_r\rrbracket$ by
\begin{gather}
H\begin{pmatrix}n_1,\dots,n_r\\a_1,\dots,a_r\\x_1,\dots,x_r\end{pmatrix}\coloneq\sum_{0<d_1<\cdots<d_r}\prod_{j=1}^re^{d_jx_j}\eta^{a_jd_j}\left(\frac{q^{d_j}}{1-q^{d_j}}\right)^{n_j},\\
g\binom{a_1,\dots,a_r}{x_1,\dots,x_r}\coloneq\sum_{k_1,\dots,k_r>0}\left(\frac{N}{-2\pi\sqrt{-1}}\right)^{k_1+\cdots+k_r}\!g\vvec{k_1,\dots,k_r}{a_1,\dots,a_r}{q}x_1^{k_1-1}\cdots x_r^{k_r-1}.
\end{gather}
\end{dfn}

\begin{lem}[\cite{Kitada23}]\label{lem:gH}
For any $a_1,\dots,a_r\in\ZZ/N\ZZ$, we have
$$g\binom{a_1,\dots,a_r}{x_1,\dots,x_r}=H\left(\begin{array}{@{}c@{}c@{}c@{}l@{}}1&,\dots,&1&,1\\a_r-a_{r-1}&,\dots,&a_2-a_1&,a_1\\x_r-x_{r-1}&,\dots,&x_2-x_1&,x_1\end{array}\right).$$
\end{lem}

\begin{proof}
By \Cref{dfn:div}, we have
\begin{align}
(\text{L.H.S})=\sum_{\substack{0<d_1<\cdots< d_r\\k_1,\dots,k_r>0\\c_1,\dots,c_r>0}}\prod_{j=1}^r\frac{(c_jx_j)^{k_j-1}}{(k_j-1)!}\eta^{a_jc_j}q^{c_jd_j}=\sum_{\substack{0<d_1<\cdots<d_r\\c_1,\dots,c_r>0}}\prod_{j=1}^re^{c_jx_j}\eta^{a_jc_j}q^{c_jd_j}.
\end{align}
On the other hand, we have
\begin{align}
(\text{R.H.S})&=\sum_{\substack{0<d_1<\cdots<d_r\\c_1,\dots,c_r>0}}\prod_{j=1}^re^{(d_{r-j+1}-d_{r-j})x_j}\eta^{(d_{r-j+1}-d_{r-j})a_j}q^{c_jd_j}.
\end{align}
Replacing
$$d_j=c_{r-j+1}^\prime+\cdots+c_r^\prime,\quad c_j=d_{r-j+1}^\prime-d_{r-j}^\prime\quad(j\in\{1,\dots,r\}),$$
we have
\begin{align}
\sum_{\substack{0<d_1<\cdots<d_r\\c_1,\dots,c_r>0}}\prod_{j=1}^re^{(d_{r-j+1}-d_{r-j})x_j}\eta^{(d_{r-j+1}-d_{r-j})a_j}q^{c_jd_j}=\sum_{\substack{c_1^\prime,\dots,c_r^\prime>0\\0<d_1^\prime<\cdots<d_r^\prime}}\prod_{j=1}^re^{c_j^\prime x_j}\eta^{a_jc_j^\prime}q^{c_j^\prime d_j^\prime}.
\end{align}
\end{proof}

Let $\mathcal{U}$ be the non-commutative polynomial ring defined by
$$\mathcal{U}\coloneq\QQ\left\langle\begin{pmatrix}n\\a\\z\end{pmatrix}\relmid n\in\ZZ_{>0},a\in\ZZ/N\ZZ,z\in X\right\rangle,$$
where the set $X$ is given by
$$X\coloneq\left\{\sum_{i>0}m_ix_i\relmid m_i\in\ZZ_{\ge0},m_i=0\text{ for almost all $i$}\right\}.$$
The concatenation, harmonic and shuffle products on $\mathcal{U}$ extend analogously to the third line of indicies.
Note that $H$ satisfies the $\ast$-product. Define the exponential map on $\mathcal{U}$ by
\begin{align}
\exp\left(\begin{pmatrix}n_1,\dots,n_r\\a_1,\dots,a_r\\z_1,\dots,z_r\end{pmatrix}\right)\coloneq\sum_{\substack{1\le m\le r\\i_1+\cdots+i_m=r\\i_1,\dots,i_m>0}}\frac{1}{i_1!\cdots i_m!}\begin{pmatrix}n^\prime_{i_1},\dots,n^\prime_{i_m}\\a^\prime_{i_1},\dots,a^\prime_{i_m}\\z^\prime_{i_1},\dots,z^\prime_{i_m}\end{pmatrix},
\end{align}
where $p^\prime_{i_k}=p_{i_1+\cdots+i_{k-1}+1}+\cdots+p_{i_1+\cdots+i_k}$ for $p\in\{n,a,z\}$ and $k=1,\dots,m$.

\begin{prop}[Hoffman \cite{Hoff00}]
$\mathcal{U}_{\sha}$ and $\mathcal{U}_\ast$ are commutative $\QQ$-algebras, and we have an isomorphism between them:
$$\exp:\mathcal{U}_{\sha}\xlongrightarrow{\sim}\mathcal{U}_{\ast}.$$
\end{prop}

\begin{dfn}[\cite{Kitada23}]
For $a_1,\dots,a_r\in\ZZ/N\ZZ$, we define $h\binom{a_1,\dots,a_r}{x_1,\dots,x_r}\in\CC\llbracket q\rrbracket\llbracket x_1,\dots,x_r\rrbracket$ by
\begin{align}
h\binom{a_1,\dots,a_r}{x_1,\dots,x_r}&\coloneq H\circ\exp\left(\begin{pmatrix}1,\dots,1\\a_1,\dots,a_r\\x_1,\dots,x_r\end{pmatrix}\right)\\
&=\sum_{\substack{1\le m\le r\\i_1+\cdots+i_m=r\\i_1,\dots,i_m>0}}\frac{1}{i_1!\cdots i_m!}H\begin{pmatrix}n^\prime_{i_1},\dots,n^\prime_{i_m}\\a^\prime_{i_1},\dots,a^\prime_{i_m}\\z^\prime_{i_1},\dots,z^\prime_{i_m}\end{pmatrix}.
\end{align}
\end{dfn}

\begin{lem}[\cite{Kitada23}]
For any $a_1,\dots,a_{r+s}\in\ZZ/N\ZZ$, we have
$$h\binom{a_1,\dots,a_r}{x_1,\dots,x_r}h\binom{a_{r+1},\dots,a_{r+s}}{x_{r+1},\dots,x_{r+s}}=\left.h\binom{a_1,\dots,a_{r+s}}{x_1,\dots,x_{r+s}}\relmid\mathrm{sh}^{(r+s)}_r\right.,$$
where
$$\mathrm{sh}^{(r+s)}_r\coloneq\sum_{\sigma\in\mathrm{Sh}^{(r+s)}_r}\sigma\in\ZZ[\mathfrak{S}_{r+s}],$$
and the action extends to an action of the group ring $\ZZ[\mathfrak{S}_{r+s}]$ by linearity.
\end{lem}

\begin{proof}
Since $H$ satisfies the harmonic product, the map
$$H:\mathcal{U}_*\to\varinjlim_{r}\CC\llbracket q\rrbracket\llbracket x_1,\dots,x_r\rrbracket:\begin{pmatrix}n_1,\dots,n_r\\a_1,\dots,a_r\\z_1,\dots,z_r\end{pmatrix}\mapsto H\begin{pmatrix}n_1,\dots,n_r\\a_1,\dots,a_r\\z_1,\dots,z_r\end{pmatrix}$$
is a homomorphism. Clearly, we have
$$h\binom{a_1,\dots,a_r}{x_1,\dots,x_r}=H\circ\exp\left(\begin{pmatrix}1,\dots,1\\a_1,\dots,a_r\\x_1,\dots,x_r\end{pmatrix}\right).$$
Therefore, $h:\mathcal{U}_\sha\to\CC\llbracket q\rrbracket\llbracket x_1,\dots,x_r\rrbracket$ is a homomorphism.
\end{proof}

The following lemma provides a characterization of the product $\tsha$ through generating functions.

\begin{lem}\label{lem:tsha}
Let $F\in\CC\llbracket q\rrbracket\llbracket x_1,\dots,x_r\rrbracket$ and $f\in\CC\llbracket q\rrbracket$ satisfying
$$F\binom{a_1,\dots,a_r}{x_1,\dots,x_r}=\sum_{k_1,\dots,k_r>0}f\binom{k_1,\dots,k_r}{a_1,\dots,a_r}x_1^{k_1-1}\cdots x_r^{k_r-1}.$$
Then, the following statements are equivalent:
\begin{enumerate}
\renewcommand{\labelenumi}{(\roman{enumi})}
\item $f\in\CC\llbracket q\rrbracket$ satisfies the $\tsha$-product i.e. it holds
$$f\binom{k_1,\dots,k_r}{a_1,\dots,a_r}f\binom{k_{r+1},\dots,k_{r+s}}{a_{r+1},\dots,a_{r+s}}=f\left(\binom{k_1,\dots,k_r}{a_1,\dots,a_r}\tsha\binom{k_{r+1},\dots,k_{r+s}}{a_{r+1},\dots,a_{r+s}}\right).$$
for any $k_1,\dots,k_{r+s}\in\ha^1$.
\item For any positive integers $r,s>0$ and $a_1,\dots,a_{r+s}\in\ZZ/N\ZZ$, it holds
$$F^\#\binom{a_1,\dots,a_r}{x_1,\dots,x_r}F^\#\binom{a_{r+1},\dots,a_{r+s}}{x_{r+1},\dots,x_{r+s}}=\left.F^\#\binom{a_1,\dots,a_{r+s}}{x_1,\dots,x_{r+s}}\relmid\mathrm{sh}^{(r+s)}_r\right..$$
Here, $F^\#\binom{a_1,\dots,a_r}{x_1,\dots,x_r}\coloneq F^\#\binom{a_1,a_1+a_2,\dots,a_1+\cdots+a_r}{x_1,x_1+x_2,\dots,x_1+\cdots+x_r}$.
\end{enumerate}
\end{lem}

\begin{proof}
Since
$$F^\#\binom{a_1,\dots,a_r}{x_1,\dots,x_r}=\sum_{k_1,\dots,k_r>0}f\left(\rho^{-1}\binom{k_1,\dots,k_1}{a_1,\dots,a_r}\right)x_1^{k_1-1}(x_1+x_2)^{k_2-1}\cdots(x_1+\cdots+x_r)^{k_r-1},$$
the statement (ii) is equivalent to
$$f\left(\rho^{-1}\binom{k_1,\dots,k_r}{a_1,\dots,a_r}\right)f\left(\rho^{-1}\binom{k_{r+1},\dots,k_{r+s}}{a_{r+1},\dots,a_{r+s}}\right)=f\left(\rho^{-1}\left(\binom{k_1,\dots,k_r}{a_1,\dots,a_r}\sha\binom{k_{r+1},\dots,k_{r+s}}{a_{r+1},\dots,a_{r+s}}\right)\right)$$
for any $k_1,\dots,k_{r+s}\ge1$ and $a_1,\dots,a_{r+s}\in\ZZ/N\ZZ$ (see Ihara--Kaneko--Zagier \cite{IKZ}, Section 8). By definition of $\tsha$ and bijectivity of $\rho$, this statement is equivalent to
$$f\binom{k_1,\dots,k_r}{a_1,\dots,a_r}f\binom{k_{r+1},\dots,k_{r+s}}{a_{r+1},\dots,a_{r+s}}=f\left(\binom{k_1,\dots,k_r}{a_1,\dots,a_r}\tsha\binom{k_{r+1},\dots,k_{r+s}}{a_{r+1},\dots,a_{r+s}}\right)$$
for any $k_1,\dots,k_{r+s}\ge1$ and $a_1,\dots,a_{r+s}\in\ZZ/N\ZZ$.
\end{proof}

\begin{dfn}[\cite{Kitada23}]
We define $g_{\tsha}\in\CC\llbracket q\rrbracket\llbracket x_1,\dots,x_r\rrbracket$ by
$$g_{\tsha}\binom{a_1,\dots,a_r}{x_1,\dots,x_r}\coloneq h\binom{a_r-a_{r-1},\dots,a_2-a_1,a_1}{x_r-x_{r-1},\dots,x_2-x_1,x_1},$$
and we define the shuffle regularization $g^{\tsha}:\ha^1\to\CC\llbracket q\rrbracket$ via the coefficients of the generating function:
$$\sum_{k_1,\dots,k_r>0}\left(\frac{N}{-2\pi\sqrt{-1}}\right)^{k_1+\cdots+k_r}\hspace{-5pt}g^{\tsha}\vvec{k_1,\dots,k_r}{a_1,\dots,a_r}{q}x_1^{k_1-1}\cdots x_r^{k_r-1}\coloneq g_{\tsha}\binom{a_1,\dots,a_r}{x_1,\dots,x_r}.$$
\end{dfn}

\begin{prop}[\cite{Kitada23}]
For any $n_1,\dots,n_{r+s}\ge1$ and $a_1,\dots,a_{r+s}\in\ZZ/N\ZZ$, we have
$$g^{\tsha}\binom{n_1,\dots,n_r}{a_1,\dots,a_r}g^{\tsha}\binom{n_{r+1},\dots,n_{r+s}}{a_{r+1},\dots,a_{r+s}}=g^{\tsha}\left(\binom{n_1,\dots,n_r}{a_1,\dots,a_r}\tsha\binom{n_{r+1},\dots,n_{r+s}}{a_{r+1},\dots,a_{r+s}}\right).$$
\end{prop}

\begin{proof}
Let $\rho_{r,s}\in\mathfrak{S}_{r+s}$, $\tau_r\in\mathfrak{S}_r$ be given by
$$\rho_{r,s}=\binom{1\;\cdots\;r\quad r+1\;\cdots\;r+s}{r\;\cdots\;1\quad r+s\;\cdots\;r+1},\quad\tau_{r}=\binom{1\;\cdots\;r}{r\;\cdots\;1}.$$
We have
\begin{align}
g_{\tsha}^\#\binom{a_1,\dots,a_r}{x_1,\dots,x_r}g_{\tsha}^\#\binom{a_{r+1},\dots,a_{r+s}}{x_{r+1},\dots,x_{r+s}}&=h\binom{a_r,\dots,a_1}{x_r,\dots,x_1}h\binom{a_{r+s},\dots,a_{r+1}}{x_{r+s},\dots,x_{r+1}}\\
&=h\binom{a_1,\dots,a_r}{x_1,\dots,x_r}h\binom{a_{r+1},\dots,a_{r+s}}{x_{r+1},\dots,x_{r+s}}\Big|\rho_{r,s}.
\end{align}
Since $h$ satisfies the $\sha$-product, we have
\begin{align}
h\binom{a_1,\dots,a_r}{x_1,\dots,x_r}h\binom{a_{r+1},\dots,a_{r+s}}{x_{r+1},\dots,x_{r+s}}\Big|\rho_{r,s}&=h\binom{a_1,\dots,a_{r+s}}{x_1,\dots,x_{r+s}}\Big|\mathrm{sh}_r^{(r+s)}\Big|\rho_{r,s}\\
&=g_{\tsha}^\#\binom{a_1,\dots,a_{r+s}}{x_1,\dots,x_{r+s}}\Big|\tau_{r+s}\Big|\mathrm{sh}_r^{(r+s)}\Big|\rho_{r,s}.
\end{align}
Since $\tau_{r+s}\sigma\rho_{r,s}\in\mathrm{Sh}^{(r+s)}_r$ for any $\sigma\in\mathrm{Sh}^{(r+s)}_r$, we have
$$g_{\tsha}^\#\binom{a_1,\dots,a_{r+s}}{x_1,\dots,x_{r+s}}\Big|\tau_{r+s}\Big|\mathrm{sh}_r^{(r+s)}\Big|\rho_{r,s}=g_{\tsha}^\#\binom{a_1,\dots,a_{r+s}}{x_1,\dots,x_{r+s}}\Big|\mathrm{sh}_r^{(r+s)}.$$
By \Cref{lem:tsha}, $g^{\tsha}$ satisfies the $\tsha$-product.
\end{proof}

In the convergent case, $g^{\tsha}$ coincides with $g$.

\begin{lem}[\cite{Kitada23}]\label{lem:gsha}
For $n_1,\dots,n_r\ge2$ and $a_1,\dots,a_r\in\ZZ/N\ZZ$, we have
$$g^{\tsha}\vvec{n_1,\dots,n_r}{a_1,\dots,a_r}{q}=g\vvec{n_1,\dots,n_r}{a_1,\dots,a_r}{q}.$$
\end{lem}

\begin{proof}
By definition of $g_{\tsha},h,H\in\CC\llbracket q\rrbracket\llbracket x_1,\dots,x_r\rrbracket$, we have
$$g_{\tsha}\binom{a_1,\dots,a_r}{x_1,\dots,x_r}=\sum_{\substack{1\le m\le r\\i_1+\cdots+i_m=r\\i_j>0}}\frac{1}{i_1!\cdots i_m!}H\left(\begin{matrix}i_1,\dots,i_m\\a_{i_1}^{\prime\prime},\dots,a_{i_m}^{\prime\prime}\\x_{i_1}^{\prime\prime},\dots,x_{i_m}^{\prime\prime}\end{matrix}\right),$$
where $p_{i_j}^{\prime\prime}\coloneq p_{r-(i_1+\cdots+i_{j-1})}-p_{r-(i_1+\cdots+i_j)}$ for $p\in\{a,x\}$ and $j=1,\dots,r$. Since the coefficients of $x_1^{n_1-1}\cdots x_r^{n_r-1}$ in $H\left(\begin{matrix}i_1,\dots,i_m\\a_{i_1}^{\prime\prime},\dots,a_{i_m}^{\prime\prime}\\x_{i_1}^{\prime\prime},\dots,x_{i_m}^{\prime\prime}\end{matrix}\right)$ are zero for any $n_1,\dots,n_r\ge2$ when $1\le m<r$. By \Cref{lem:gH}, we have
\begin{align}
&\text{The coefficient of $x_1^{n_1-1}\cdots x_r^{n_r-1}$ in }g_{\tsha}\binom{a_1,\dots,a_r}{x_1,\dots,x_r}\\
&=\text{The coefficient of $x_1^{n_1-1}\cdots x_r^{n_r-1}$ in }H\left(\begin{array}{@{}c@{}c@{}c@{}l@{}}1&,\dots,&1&,1\\a_r-a_{r-1}&,\dots,&a_2-a_1&,a_1\\x_r-x_{r-1}&,\dots,&x_2-x_1&,x_1\end{array}\right)\\
&=\text{The coefficient of $x_1^{n_1-1}\cdots x_r^{n_r-1}$ in }g\binom{a_1,\dots,a_r}{x_1,\dots,x_r}.
\end{align}
\end{proof}

\subsection{Shuffle regularization for MES of level $N$}

\begin{dfn}
We define the shuffle regularized MES of level $N$ $G:\ha^1\to\mathcal{O}(\HH)$ by
$$G^{\tsha}(w;\tau)\coloneq(\zeta^{\tsha}\star g^{\tsha})(w;q)$$
for $w\in\ha^1$.
\end{dfn}

\begin{rem}\label{rem:MESshu}
$G^{\tsha}$ satisfy the shuffle product since all maps $\Delta$, $\zeta^{\tsha}$ and $g^{\tsha}$ are $\tsha$-homomorphisms.
\end{rem}

This regularization makes sense, in other words, the regularized MES are equal to the original MES for the cases of convergence.

\begin{prop}\label{prop:MESha2}
It holds $G^{\tsha}=G$ on $\ha^2$.
\end{prop}

\begin{proof}
By the explicit formula for the Fourier expansion of MES, we have
\begin{align}
&(G^{\tsha}-G)\vvec{n_1,\dots,n_r}{a_1,\dots,a_r}{\tau}\\
&=\sum_{\substack{0\le h\le r\\0<t_1<\cdots<t_h<r+1}}\zeta\binom{n_1,\dots,n_{t_1-1}}{a_1,\dots,a_{t_1-1}}\rising{5}{$\sum_{\substack{t_1\le q_1\le t_2-1\\[-5pt]\vdots\\t_h\le q_h\le r}}$}\rising{5}{$\sum_{\substack{k_{t_j}+\cdots+k_{t_{j+1}-1}\\=n_{t_j}+\cdots+n_{t_{j+1}-1}\\(1\le j\le h),k_i\ge1}}$}\prod_{j=1}^h\Bigg\{(-1)^{l_j}\Bigg(\prod_{\substack{p=t_j\\p\ne q_j}}^{t_{j+1}-1}\binom{k_p-1}{n_p-1}\Bigg)\\
&\times\zeta\bigg(\begin{array}{@{}c@{}c@{}c@{}}k_{q_j-1}&,\dots,&k_{t_j}\\a_{q_j}-a_{q_j-1}&,\dots,&a_{q_j}-a_{t_j}\end{array}\bigg)\zeta\bigg(\begin{array}{@{}c@{}c@{}c@{}}k_{q_j+1}&,\dots,&k_{t_{j+1}-1}\\a_{q_j+1}-a_{q_j}&,\dots,&a_{t_{j+1}-1}-a_{q_j}\end{array}\bigg)\Bigg\}
(g^{\tsha}-g)\bigg(\begin{matrix}k_{q_1},\dots,k_{q_h}\\a_{q_1},\dots,a_{q_h}\end{matrix};q\bigg)
\end{align}
for $n_1,\dots,n_r\ge2$. The terms with $k_{q_1},\dots,k_{q_h}\ge2$ vanish by \Cref{lem:gsha}. When $k_{q_{j_1}},\dots,k_{q_{j_s}}=1$ for some $1\le j_1<\cdots<j_s\le h$, the terms vanish by \Cref{cor:antizeta} since we can write $(g^{\tsha}-g)\vvec{k_{q_1},\dots,k_{q_h}}{a_{q_1},\dots,a_{q_h}}{q}=\sum_{l=1}^sf_l(q)$ such that $f_l$ does not depend on $a_{q_{j_l}}$. 
\end{proof}

\section{Linear relations among regularized MES of level $N$}\label{sec:rel}
In this final section, we obtain the restricted double shuffle relations and the distribution relations for MES. We also provide the sum and weighted sum formulas for double Eisenstein series (DES).
\newtheorem*{main2}{\Cref{thm:main2}}
\begin{main2}[Restricted double shuffle relation]
For any words $w_1,w_2\in \ha^2$, we have
\begin{align}\label{eq:rDSR}
G(w_1\tas w_2;\tau)=G^{\tsha}(w_1\tsha w_2;\tau).
\end{align}
\end{main2}

\begin{proof}
By \Cref{prop:MESha2}, it holds that $G(w;\tau)=G^{\tsha}(w;\tau)$ for any $w\in\ha^2$.
Therefore, the statement follows since $G(w;\tau)=G^{\tsha}(w;\tau)$ satisfy both the harmonic product and the shuffle product formulas (\Cref{rem:MEShar}, \Cref{rem:MESshu}).
\end{proof}

Let $G^\sha_{n_1,\dots,n_r}(\tau)$, $\zeta^\sha(n_1,\dots,n_r)$ and $g^\sha_{n_1,\dots,n_r}(q)$ denote those of level $1$. Note that the $\tsha$-product is equal to the $\sha$-products when $N=1$.

\begin{thm}[Distribution relation]
For $n_1,\dots,n_r\ge1$, we have
$$\sum_{a_1,\dots,a_r\in\ZZ/N\ZZ}G^{\tsha}\bigg(\begin{matrix}n_1,\dots,n_r\\a_1,\dots,a_r\end{matrix};\tau\bigg)=G^\sha_{n_1,\dots,n_r}(N\tau).$$
\end{thm}

\begin{proof}
It suffices to show $\zeta^{\tsha}$ and $g^{\tsha}$ satisfy distribution relations i.e.
\begin{gather}
\sum_{a_1,\dots,a_r\in\ZZ/N\ZZ}\zeta^{\tsha}\binom{n_1,\dots,n_r}{a_1,\dots,a_r}=\zeta^\sha(n_1,\dots,n_r),\\
\sum_{a_1,\dots,a_r\in\ZZ/N\ZZ}g^{\tsha}\bigg(\begin{matrix}n_1,\dots,n_r\\a_1,\dots,a_r\end{matrix};q\bigg)=g^{\sha}_{n_1,\dots,n_r}(q^N).
\end{gather}
By definition of $\zeta^{\tsha}$, we have
$$\sum_{a_1,\dots,a_r\in\ZZ/N\ZZ}\zeta^{\tsha}\binom{n_1,\dots,n_r}{a_1,\dots,a_r}=\frac{1}{N^r}\sum_{b_1,\dots,b_r\in\ZZ/N\ZZ}\left(\sum_{a_1,\dots,a_r\in\ZZ/N\ZZ}\eta^{-\rho(\bm{a})\cdot\bm{b}}\right)L_{\sha}^{\mathrm{reg}}\binom{n_1,\dots,n_r}{b_1,\dots,b_r}\bigg|_{T=0}.$$
Since
$$\sum_{a_1,\dots,a_r\in\ZZ/N\ZZ}\eta^{-\rho(\bm{a})\cdot\bm{b}}=\prod_{j=1}^r\sum_{a_j\in\ZZ/N\ZZ}\eta^{a_j(b_j-b_{j+1})}=\begin{cases}N^r&b_1=\cdots=b_r=0,\\0&\mathrm{otherwise}.\end{cases},$$
we have
$$\sum_{a_1,\dots,a_r\in\ZZ/N\ZZ}\zeta^{\tsha}\binom{n_1,\dots,n_r}{a_1,\dots,a_r}=L_{\sha}^{\mathrm{reg}}\binom{n_1,\dots,n_r}{0,\dots,0}\bigg|_{T=0}=\zeta^{\sha}(n_1,\dots,n_r).$$
By definition of $a_{i_j}^{\prime\prime}$, we have
\begin{align}
\sum_{a_1,\dots,a_r\in\ZZ/N\ZZ}g_{\tsha}\vvec{a_1,\dots,a_r}{x_1,\dots,x_r}{q}&=\sum_{\substack{1\le m\le r\\i_1+\cdots+i_m=r\\i_1,\dots,i_m>0}}\frac{1}{i_1!\cdots i_m!}\sum_{a_1,\dots,a_r\in\ZZ/N\ZZ}H\left(\begin{matrix}i_1,\dots,i_m\\a_{i_1}^{\prime\prime},\dots,a_{i_m}^{\prime\prime}\\x_{i_1}^{\prime\prime},\dots,x_{i_m}^{\prime\prime}\end{matrix};q\right)\\
&=\sum_{\substack{1\le m\le r\\i_1+\cdots+i_m=r\\i_1,\dots,i_m>0}}\frac{1}{i_1!\cdots i_m!}\Bigg(\sum_{\substack{a_1,\dots,a_r\in\ZZ/N\ZZ\\\widehat{a_{i_1}},\dots,\widehat{a_{i_m}}}}1\Bigg)\\
&\times\sum_{a_{i_1},\dots,a_{i_m}\in\ZZ/N\ZZ}H\left(\begin{matrix}i_1,\dots,i_m\\a_{i_1},\dots,a_{i_m}\\x_{i_1}^{\prime\prime},\dots,x_{i_m}^{\prime\prime}\end{matrix};q\right).
\end{align}
Therefore, we have
\begin{align}
\sum_{a_1,\dots,a_r\in\ZZ/N\ZZ}H\left(\begin{matrix}n_1,\dots,n_r\\a_1,\dots,a_r\\x_1,\dots,x_r\end{matrix};q\right)&=\sum_{0<d_1<\cdots<d_r}\prod_{j=1}^r\left(\sum_{a_j\in\ZZ/N\ZZ}\eta^{a_jd_j}\right)e^{d_jx_j}\left(\frac{q^{d_j}}{1-q^{d_j}}\right)^{n_j}\\
&=N^r\sum_{\substack{0<d_1<\cdots<d_r\\d_i=0}}\prod_{j=1}^re^{d_jx_j}\left(\frac{q^{d_j}}{1-q^{d_j}}\right)^{n_j}\\
&=N^rH\left(\begin{matrix}n_1,\dots,n_r\\0,\dots,0\\Nx_1,\dots,Nx_r\end{matrix};q^N\right).
\end{align}
By definition of $g_{\tsha}$, we have
\begin{align}
\sum_{a_1,\dots,a_r\in\ZZ/N\ZZ}g_{\tsha}\vvec{a_1,\dots,a_r}{x_1,\dots,x_r}{q}&=\sum_{\substack{1\le m\le r\\i_1+\cdots+i_m=r\\i_1,\dots,i_m>0}}\frac{1}{i_1!\cdots i_m!}N^{r-m}\cdot N^mH\left(\begin{matrix}i_1,\dots,i_m\\0,\dots,0\\Nx_{i_1}^{\prime\prime},\dots,Nx_{i_m}^{\prime\prime}\end{matrix};q^N\right)\\
&=N^rg_{\tsha}\vvec{0,\dots,0}{Nx_1,\dots,Nx_r}{q^N}.
\end{align}
Comparing both coefficients, we have
$$\sum_{a_1,\dots,a_r\in\ZZ/N\ZZ}g^{\tsha}\bigg(\begin{matrix}n_1,\dots,n_r\\a_1,\dots,a_r\end{matrix};q\bigg)=g^{\sha}_{n_1,\dots,n_r}(q^N).$$
\end{proof}

We give sum and weighted sum formula for DES in terms of the generating functions. Let $F^k_{a_1,a_2}(x_1,x_2)$ be the generating function of DES of weight $k$ $(\ge4)$ and level $N$,
$$F^k_{a_1,a_2}(x_1,x_2)\coloneq\sum_{\substack{i+j=k\\i,j>1}}G\vvec{i,j}{a_1,a_2}{\tau}x_1^{i-1}x_2^{j-1}\quad(a_1,a_2\in\ZZ/N\ZZ).$$
By the restricted double shuffle relations for DES, we have the following equations for the generating functions.

\begin{lem}[\cite{Kitada23}]\label{lem:genG}
For any integer $k\ge4$ and $a_1,a_2\in\ZZ/N\ZZ$, we have
\begin{align}
&F^k_{a_1,a_2}(x_1,x_2)+F^k_{a_2,a_1}(x_2,x_1)+\delta_{a_1,a_2}G\binom{k}{a_1}\left(\frac{x_1^{k-1}-x_2^{k-1}}{x_1-x_2}-(x_1^{k-2}-x_2^{k-2})\right)\\
&=F^k_{a_1,a_1+a_2}(x_1,x_1+x_2)+F^k_{a_2,a_1+a_2}(x_2,x_1+x_2)\\
&+\left(G^{\tsha}\binom{1,k-1}{a_1,a_1+a_2}+G^{\tsha}\binom{1,k-1}{a_2,a_1+a_2}\right)(x_1+x_2)^{k-2}\\
&-\sum_{\substack{i+j=k\\i>0,j>1}}G^{\tsha}\binom{i,j}{a_1,a_1+a_2}x_1^{k-2}-\sum_{\substack{i+j=k\\i>0,j>1}}G^{\tsha}\binom{i,j}{a_2,a_1+a_2}x_2^{k-2}.
\end{align}
\end{lem}

\begin{proof}
By the restricted double shuffle relations, we have
\begin{align}
G\binom{i}{a_1}G\binom{j}{a_2}&=G\binom{i,j}{a_1,a_2}+G\binom{j,i}{a_2,a_1}+\delta_{a_1,a_2}G\binom{i+j}{a_1}\\
&=\sum_{\substack{m+n=i+j\\m,n>1}}\left\{\binom{n-1}{j-1}G\binom{m,n}{a_1,a_1+a_2}+\binom{n-1}{i-1}G\binom{m,n}{a_2,a_1+a_2}\right\}\\
&+\binom{i+j-2}{i-1}\left(G^{\tsha}\binom{1,i+j-1}{a_1,a_1+a_2}+G^{\tsha}\binom{1,i+j-1}{a_2,a_1+a_2}\right)
\end{align}
for $i,j>1$. Multiplying $x_1^{i-1}x_2^{j-1}$ and adding up for $i+j=k$, $i,j>1$, we obtain the equation.
\end{proof}

By using this lemma, we have sum and weighted sum formulas for DES.

\begin{thm}[Sum formula for DES, Kitada \cite{Kitada23}]
For any even integer $k\ge4$ and $a\in\ZZ/N\ZZ$, we have
$$2\sum_{\substack{i+j=k\\i,j>1}}\left((-1)^{i-1}G\binom{i,j}{a,a}+G\binom{i,j}{a,2a}\right)+4G^{\tsha}\binom{1,k-1}{a,2a}=G\binom{k}{a}.$$
\end{thm}

\begin{proof}
It follows by inserting $k$ even, $(x_1,x_2)=(1,-1)$ and $a_1=a_2=a$ in \Cref{lem:genG}.
\end{proof}

\begin{thm}[Weighted sum formula for DES]
For any integer $k\ge4$ and $a\in\ZZ/N\ZZ$, we have
$$\sum_{\substack{i+j=k\\i,j>0}}\left\{(2^{j-1}-1)G^{\tsha}\binom{i,j}{a,2a}+(1-\delta_{j,1})G^{\tsha}\binom{i,j}{a,a}\right\}=\frac{k-3}{2}G^{\tsha}\binom{k}{a}.$$
\end{thm}

\begin{proof}
It follows by inserting $(x_1,x_2)=(1,1)$ and $a_1=a_2=a$ in \Cref{lem:genG}.
\end{proof}

\bibliographystyle{plain}
\bibliography{references}

\end{document}